\documentclass[reqno,12pt]{amsart}
\usepackage{amsmath,amsthm,amssymb,amsfonts,amscd}
\usepackage{epsfig}
\usepackage[all]{xy}
\usepackage{tabularx}  
\usepackage{booktabs}  
\usepackage{float}

\numberwithin{equation}{section}

\theoremstyle{plain}
\newtheorem{theorem}{Theorem}

\newtheorem{conjecture}[theorem]{Conjecture}
\newtheorem{lemma}[theorem]{Lemma}
\newtheorem{corollary}[theorem]{Corollary}
\newtheorem{proposition}[theorem]{Proposition}
\newtheorem*{theorem*}{Theorem}
\newtheorem*{conjecture*}{Conjecture}

\theoremstyle{definition}
\newtheorem{remark}[theorem]{Remark}

\newtheorem{definition}[theorem]{Definition}

\newcommand{\CC}{{\mathbb{C}}}

\newcommand{\QQ}{{\mathbb{Q}}}

\newcommand{\ZZ}{{\mathbb{Z}}}

\newcommand{\SL}{{\mathrm{SL}}}
\newcommand{\epi}{{\bf e}}

\newcommand{\Hom}{\mathrm{Hom}}

\newcommand{\Jac}{\mathrm{Jac}}
\newcommand{\Fix}{\mathrm{Fix}}

\newcommand{\Aut}{\mathrm{Aut}}
\newcommand{\id}{\mathrm{id}}

\newcommand{\bx}{{\bf x}}
\newcommand{\hess}{\mathrm{hess}}
\newcommand{\age}{\mathrm{age}}

\newcommand{\QR}[2]{
\left.\raisebox{0.5ex}{\ensuremath{#1}}
\ensuremath{\mkern-3mu}\middle/\ensuremath{\mkern-3mu}
\raisebox{-0.5ex}{\ensuremath{#2}}\right.}

\def\A{{\mathcal A}}

\def\C{{\mathcal C}}

\def\T{{\mathcal T}}

\begin{document}
\title{Orbifold Jacobian algebras for invertible polynomials}
\date{\today}
\author{Alexey Basalaev}
\address{Universit\"at Mannheim, Lehrsthul f\"ur Mathematik VI, Seminargeb\"aude A 5, 6, 68131 Mannheim, Germany}
\email{basalaev@uni-mannheim.de}
\author{Atsushi Takahashi}
\address{Department of Mathematics, Graduate School of Science, Osaka University, 
Toyonaka Osaka, 560-0043, Japan}
\email{takahashi@math.sci.osaka-u.ac.jp}
\author{Elisabeth Werner}
\address{Leibniz Universit\"at Hannover, Welfengarten 1, 30167 Hannover, Germany}
\email{werner.elisabeth@math.uni-hannover.de }
\begin{abstract}
An important invariant of a polynomial $f$ is its Jacobian algebra  defined by its partial derivatives.
Let $f$ be invariant with respect to the action of a finite group of  diagonal symmetries $G$.
We axiomatically define an orbifold Jacobian $\ZZ/2\ZZ$-graded algebra for the pair $(f,G)$ 
and show its existence and uniqueness in the case, when $f$ is an invertible polynomial.
In case when $f$ defines an ADE singularity, we illustrate its geometric meaning.
\end{abstract}
\maketitle
\tableofcontents
\section{Introduction}

Let $f=f(\bx)=f(x_1,\dots, x_N)\in\CC[x_1,\dots, x_N]$ be a polynomial such that 
the {\em Jacobian algebra} $\Jac(f):= \CC[x_1,\dots, x_N]/(\partial f/\partial x_1,\dots,\partial f/\partial x_N)$ of $f$ 
is a finite dimensional $\CC$-algebra.
In this paper, we shall give axioms which should characterize a generalization of the Jacobian algebra $\Jac(f)$ of $f$ 
for the pair $(f,G)$ where $G$ is a finite abelian group acting diagonally on variables which respects $f$.

Such a pair $(f,G)$, often called a {\em Landau-Ginzburg orbifold}, has been studied intensively by many mathematicians and physicists working in mirror symmetry
for more than twenty years since it yields important, interesting and unexpected geometric information. 
In particular, the so-called orbifold construction of a mirror manifold from a Calabi-Yau hypersurface is the cornerstone.

Certain works towards the definition of the Frobenius algebras associated to the pair $(f,G)$ were 
also done previously by R. Kaufmann and M. Krawitz.
In \cite{K1}, R.Kaufmann proposes a general construction of the {\em orbifolded Frobenius superalgebra} of $(f,G)$. 
In order to build such a $\ZZ/2\ZZ$-graded algebra, 
one should make a certain non-unique choice called a ``choice of a two cocycle''. 
A different choice of this cocycle gives indeed a different product structure. 
This construction was later used by Kaufmann in \cite{K2} for the mirror symmetry purposes from the point of view of physics.
In \cite{Kr}, M. Krawitz proposes a very special construction of a commutative (not a $\ZZ/2\ZZ$-graded) 
algebra, for the pair $(f,G)$. 
Later enhanced, in \cite{FJJS} this definition was used to set up the mirror symmetry on the level of Frobenius algebras. 
However, the crucial part of it remained to be the particularly fixed product that could only be well-defined 
for weighted-homogeneous polynomials. There is also no explanation why a particular product structure is chosen.

Mirror symmetry on the level of Frobenius algebras is a first step towards the mirror isomorphism of Frobenius manifolds where the key role is played by the so-called {\em primitive form}.
From the point of view of mirror symmetry, the algebras we consider here are those 
in the complex geometry side, the so-called B-model side. 
The major advantage of our work comparing to that of Kaufmann and Krawitz is that our construction works as a starting point for the mirror symmetry on the level of Frobenius manifolds having the notion of a primitive form (cf. \cite{S1,S2,ST})
in the definition (cf. the role of $\zeta$ in Definition~\ref{axioms}). 
The second important point is the following. In both Kaufmann and Krawitz constructions, one predefines the product either by the definition or by the special choice of a two cocycle. 
In our axiomatization, we do not this and hence we are able to study our algebras even for non-weighted 
homogeneous polynomials, such as cusp polynomials, for which the mirror isomorphism of Frobenius manifolds 
can be proven under the assumption of the existence of the Saito theory for the pair $(f,G)$ \cite{BTW2}. 
The last but not least is that our algebra
inherits a natural $\ZZ/2\ZZ$-grading from the Hodge theory associated to $(f,G)$. 
This $\ZZ/2\ZZ$-grading appears only in an abstract way in the definition of Kaufmann and was not considered at all by Krawitz.

Let us comment in detail on the first point. It is well-known that $\Jac(f)$ has a structure of a Frobenius algebra (see \cite{AGV85}).
Namely, by a choice of a nowhere vanishing holomorphic $N$-form, there is 
an isomorphism $\Jac(f) \cong \Omega_f := \Omega^N(\CC^N)/(df \wedge \Omega^{N-1}(\CC^N))$.
It is on $\Omega_f$, where a natural or canonical non-degenerate symmetric bilinear form, called {\em the residue pairing}, exists.
As a result, the above isomorphism equips $\Jac(f)$ with an induced bilinear form.
Therefore, even if the group $G$ is trivial, it is important to consider a pair $(\Jac(f),\Omega_f)$. 
This pair can also be considered as an example of a pair $\left(H\!H^\bullet(\C), H\!H_\bullet(\C)\right)$, consisting of the Hochschild cohomology and the Hochschild homology of a suitable dg- or $\A_\infty$-category, which has a rich algebraic structure.

In this paper, we shall first introduce a $G$-twisted version of the vector space $\Omega_f$, which is denoted by $\Omega'_{f,G}$. 
This is a $\ZZ/2\ZZ$-graded vector space, which also has a $G$-grading, 
with a natural non-degenerate bilinear form called the {\em orbifold residue pairing},
a natural generalization of the residue pairing on $\Omega_f$.
Then, the $G$-twisted version of the Jacobian algebra, denoted by $\Jac'(f,G)$, will be introduced axiomatically 
as a part of the structure of the pair $(\Jac'(f,G), \Omega'_{f,G})$ 
in the way it is in the classical situation when the group $G$ is trivial.
As a result, the algebra $\Jac'(f,G)$ inherits many structures, defined naturally on $\Omega'_{f,G}$, such as 
a $\ZZ/2\ZZ$-grading, a $G$-grading, 
equivariance with respect to automorphisms of the pair $(f,G)$, the orbifold residue pairing and so on.
Our axiomatization of a $G$-twisted Jacobian algebra lists a minimum conditions to be satisfied, 
in particular, we do not prescribe the product structure. 

The expected Jacobian algebra $\Jac(f,G)$ for the pair $(f,G)$, which we shall call {\em the orbifold Jacobian algebra} of $(f,G)$, 
will be given as the $G$-invariant subalgebra of the $G$-twisted version $\Jac'(f,G)$. 
However, it is not clear in general whether such an algebra as $\Jac'(f,G)$ exists or not. 
Even if it exists it may not be unique. 

The main result of this paper is the existence and the uniqueness of the $G$-twisted Jacobian algebra $\Jac'(f,G)$
for an invertible polynomial $f$ with a subgroup $G$ of the maximal abelian symmetry group $G_f$ (Theorem~\ref{theorem_N}). 
Namely, it is uniquely determined up to isomorphism by our axiomatization.
Moreover, we show that if $G$ is a subgroup of $\SL(N;\CC)$ then the orbifold Jacobian algebra $\Jac(f,G)$ 
has a structure of a $\ZZ/2\ZZ$-graded commutative Frobenius algebra.

Another interesting result of ours (Theorem~\ref{thm_ADEiso}) concerns the case when $f$ is an invertible polynomial
giving a singularity of ADE-type and $G$ is a subgroup of $G_f \cap \SL(N;\CC)$. 
We show that in this case our orbifold Jacobian algebra $\Jac(f,G)$ is isomorphic to the usual Jacobian algebra 
This result complies with the results of \cite{et}, where concerning a crepant resolution $\widehat{\CC^3/G}$ of $\CC^3/G$,
it is shown that the geometry of vanishing cycles for a holomorphic map 
$\widehat f:\widehat{\CC^3/G}\longrightarrow \CC$ associated to $f$ is equivalent to the one for the polynomial $\overline f$. 
Therefore, our orbifold Jacobian algebra is not only natural from the view point of algebra but also from the view point 
of geometry. 
\begin{sloppypar}

{\bf Acknowledgements}.\
The first named author is partially supported by the DGF grant He2287/4--1 (SISYPH).
The second named author is supported by JSPS KAKENHI Grant Number JP16H06337, JP26610008.
We are grateful to Wolfgang Ebeling for fruitful discussions.
\end{sloppypar}
\section{Preliminaries}
\begin{definition}\label{def:1}
Let $n$ be a non-negative integer and $f=f(\bx)=f(x_1,\dots, x_n)\in\CC[x_1,\dots, x_n]$ a polynomial.
\begin{enumerate}
\item
The {\em Jacobian algebra} $\Jac(f)$ of $f$ is a $\CC$-algebra defined as 
\begin{equation}
\Jac(f)= \QR{\CC[x_1,\dots, x_n]}{\left(\frac{\partial f}{\partial x_1},\dots,\frac{\partial f}{\partial x_n}\right)}.
\end{equation}
If $\Jac(f)$ is a finite-dimensional $\CC$-algebra, then set $\mu_f:=\dim_\CC\Jac(f)$ and call it the {\em Milnor number} of $f$. 
In particular, if $n=0$ then $\Jac(f)=\CC$ and $\mu_f=1$.
\item
The {\em Hessian} of $f$ is defined as 
\begin{equation}
\hess(f):=\det \left(\frac{\partial^2 f}{\partial x_{i} \partial x_{j}}\right)_{i,j=1,\dots,n}.
\end{equation}
In particular, if $n=0$ then $\hess(f)=1$.
\end{enumerate}
\end{definition}
Throughout this paper, we denote by $N$ a positive integer and by $f=f(\bx)=f(x_1,\dots, x_N)\in\CC[x_1,\dots, x_N]$
a polynomial such that the Jacobian algebra $\Jac(f)$ of $f$
is a finite-dimensional $\CC$-algebra, unless otherwise stated. 
Let $\Omega^p(\CC^N)$ be the $\CC$-module 
of regular $p$-forms on $\CC^N$. 
Consider the $\CC$-module 
\begin{equation}
\Omega_f := \Omega^N(\CC^N)/(df \wedge \Omega^{N-1}(\CC^N)).
\end{equation}
Note that $\Omega_{f}$ is naturally a free $\Jac(f)$-module of rank one, namely, 
by choosing a nowhere vanishing $N$-form $\widetilde{\omega}\in \Omega^N(\CC^N)$ we have the following isomorphism
\begin{equation}\label{eq:isom}
\Jac(f)\stackrel{\cong}{\longrightarrow }\Omega_f,\quad 
[\phi(\bx)]\mapsto [\phi(\bx)]\omega:=[\phi(\bx)\widetilde{\omega}],
\end{equation}
where $\omega:=[\widetilde{\omega}]$ is the residue class of $\widetilde{\omega}$ in $\Omega_f$.
\begin{remark}
Such a class $\omega\in \Omega_f$ giving the isomorphism~\eqref{eq:isom} is a non-zero constant multiple of 
the residue class of $dx_1\wedge \dots\wedge dx_N$.
\end{remark}
\begin{proposition}[cf. Section I.5.11 \cite{AGV85}]
Define a $\CC$-bilinear form $J_{f}:\Omega_{f}\otimes_\CC\Omega_{f}\longrightarrow \CC$ as 
\begin{equation}
J_{f}\left(\omega_1,\omega_2\right):={\rm Res}_{\CC^N}\left[
\begin{gathered}
\phi(\bx) \psi(\bx) dx_1\wedge\dots\wedge dx_N\\
\frac{\partial f}{\partial x_1}\dots \frac{\partial f}{\partial x_N}
\end{gathered}
\right],
\end{equation}
where $\omega_1=[\phi(\bx) dx_1\wedge\dots\wedge dx_N]$ and $\omega_2=[\psi(\bx) dx_1\wedge\dots\wedge dx_N]$.
Then, the bilinear form $J_{f}$ on $\Omega_f$ is non-degenerate. 
Moreover,  for $\phi(\bx)\in\CC[x_1,\dots, x_N]$,
\begin{equation}
J_f([\phi(\bx) dx_1\wedge\dots\wedge dx_N],[\hess(f)dx_1\wedge\dots\wedge dx_N]) \ne 0
\end{equation}
if and only if $\phi({\bf 0})\ne 0$. In particular, we have
\begin{equation}
J_f([dx_1\wedge\dots\wedge dx_N],[\hess(f)dx_1\wedge\dots\wedge dx_N]) =\mu_f.
\end{equation}
\end{proposition}
Under the isomorphism~\eqref{eq:isom}, the residue pairing endows the Jacobian algebra $\Jac(f)$ with a structure of a Frobenius algebra.

\begin{definition}
An associative $\CC$-algebra $(A,\circ)$ is called {\em Frobenius} if there is a non-degenerate bilinear form 
$\eta: A \otimes A \to \CC$ such that $\eta\left( X \circ Y, Z\right) = \eta\left(X ,Y\circ Z\right)$ for $X,Y,Z\in A$.
\end{definition}

\begin{definition}
The {\em group of maximal diagonal symmetries} of $f$ is defined as 
\begin{equation}
G_f := \left\{ (\lambda_1, \ldots , \lambda_N )\in(\CC^\ast)^N \, \left| \, f(\lambda_1x_1, \ldots, \lambda_N x_N) = 
f(x_1, \ldots, x_N) \right\} \right..
\end{equation}
We shall always identify $G_f$ with the subgroup of diagonal matrices of ${\rm GL}(N;\CC)$.
Set 
\begin{equation}
G_f^{\rm SL} := G_f\cap {\rm SL}(N;\CC).
\end{equation}

\end{definition}
\begin{remark}
For a finite subgroup $G \subseteq G_f$, the pair $(f,G)$ is often called a {\em Landau-Ginzburg orbifold}. 
\end{remark}
From now on, we shall denote by $G$ a finite subgroup of $G_f$ unless otherwise stated.
In what follows, define also $\epi\left[\alpha\right] := \exp(2 \pi \sqrt{-1} \alpha)$.
The group $G$ acts naturally on $\CC^N$ and each element $g\in G$ has a unique expression of the form
\begin{equation}\label{G-notation}
g={\rm diag}\left(\epi\left[\frac{a_1}{r}\right], \dots,\epi\left[\frac{a_N}{r}\right]\right)
\quad \mbox{with } 0 \leq a_i < r,
\end{equation}
where $r$ is the order of $g$. We use the notation $(a_1/r,\dots , a_N/r)$ or
$\frac{1}{r}(a_1,\dots , a_N)$ for the element $g$.
The {\em age} of $g$, which is introduced in \cite{IR}, is defined as the rational number 
\begin{equation}
{\rm age}(g) := \frac{1}{r}\sum_{i=1}^N a_i. 
\end{equation}
Note that if $g\in G_f^\SL$ then ${\rm age}(g)\in\ZZ$.
For each $g\in G$, we denote by $\Fix(g):=\{\bx\in\CC^N~|~g\cdot \bx=\bx \}$ the fixed locus of $g$, 
by $N_g: = \dim_\CC \Fix(g)$ its dimension and by $f^g:=f|_{\Fix(g)}$ the restriction of $f$ to the fixed locus of $g$. 
Note that since $G$ acts diagonally on $\CC^N$, $\Fix(g)$ is a linear subspace of $\CC^N$. 
\begin{proposition}[cf. Proposition~5 in \cite{et13}]
For each $g\in G$, we have a natural surjective $\CC$-algebra homomorphism $\Jac(f)\longrightarrow \Jac(f^g)$.
In particular, the Jacobian algebra $\Jac(f^g)$ is also finite dimensional.
\end{proposition}
\begin{proof}
We may assume that $\Fix(g)=\{\bx\in\CC^N~|~x_{N_g+1}=\dots =x_{N}=0\}$ by a suitable renumbering of indices.
Since $f$ is invariant under $G$, $g\cdot x_i\ne x_i$ for $i=N_g+1,\dots, N$ and 
$\frac{\partial f}{\partial x_{N_g+1}},\dots, \frac{\partial f}{\partial x_N}$ form a regular sequence, we have
\[
\left(\frac{\partial f}{\partial x_{N_g+1}},\dots, \frac{\partial f}{\partial x_N}\right)\subset 
\left(x_{N_g+1},\dots, x_N\right).
\]
Therefore, we have a natural surjective $\CC$-algebra homomorphism
\begin{eqnarray*}
\Jac(f)&=&\QR{\CC[x_1,\dots, x_N]}{\left(\frac{\partial f}{\partial x_1},\dots,\frac{\partial f}{\partial x_N}\right)}\\
&\longrightarrow &
\QR{\CC[x_1,\dots, x_N]}{\left(\frac{\partial f}{\partial x_1},\dots,\frac{\partial f}{\partial x_{N_g}}, x_{N_g+1},\dots, x_N\right)}\\
&=&\QR{\CC[x_1,\dots, x_{N_g}]}{\left(\frac{\partial f^g}{\partial x_1},\dots,\frac{\partial f^g}{\partial x_{N_g}}\right)}=\Jac(f^g).
\end{eqnarray*}
\end{proof}
\begin{corollary}\label{corollary: Jac module on omega fg}
For each $g\in G$, $\Omega_{f^g}$ is naturally equipped with a structure of $\Jac(f)$-module.
\end{corollary}
\begin{proof}
Since $\Omega_{f^g}$ is a free $\Jac(f^g)$-module of rank one  (cf. \eqref{eq:isom}),
the surjective $\CC$-algebra homomorphism $\Jac(f)\longrightarrow \Jac(f^g)$ yields the statement. 
\end{proof}
\section{Orbifold Jacobian algebras}
In order to introduce an orbifold Jacobian algebra of the pair $(f,G)$, 
we first define axiomatically a $G$-twisted Jacobian algebra of $f$. 
\subsection{Setup}
\begin{definition}\label{def_Omega'}
Define a $\ZZ/2\ZZ$-graded $\CC$-module $\Omega'_{f,G}=\left(\Omega'_{f,G}\right)_{\overline{0}}\oplus \left(\Omega'_{f,G}\right)_{\overline{1}}$, $\overline i\in\ZZ/2\ZZ$, by 
\begin{subequations}
\begin{equation}
\left(\Omega'_{f,G}\right)_{\overline{0}}:=\bigoplus_{\substack{g\in G\\ N-N_g\equiv 0\ (\text{\rm mod } 2)}}\Omega'_{f,g},\quad 
\left(\Omega'_{f,G}\right)_{\overline{1}}:=\bigoplus_{\substack{g\in G\\ N-N_g\equiv 1\ (\text{\rm mod } 2)}}\Omega'_{f,g},
\end{equation}
\begin{equation}
\Omega'_{f,g}:=\Omega_{f^g}.
\end{equation}
Here, for each $g\in G$ with $\Fix(g)=\{{\bf 0}\}$, 
\begin{equation}
\Omega_{f^g}=\Omega^0(\{{\bf 0}\})/(df^g \wedge \Omega^{-1}(\{{\bf 0}\}))=\Omega^0(\{{\bf 0}\})
\end{equation}
\end{subequations}
is the $\CC$-module of rank one consisting of constant functions on $\{{\bf 0}\}$.
\end{definition}
Since the group $G$ acts on each $\Omega_{f^g}$ by the pull-back of forms via its action on $\Fix(g)$, we can define 
the following $\ZZ/2\ZZ$-graded $\CC$-module.
\begin{definition}\label{def_Omega}
Define a $\ZZ/2\ZZ$-graded $\CC$-module $\Omega_{f,G}$ as the $G$-invariant part of $\Omega'_{f,G}$,
\begin{equation}
\Omega_{f,G}=\left(\Omega'_{f,G}\right)^G.
\end{equation}
That is, $\Omega_{f,G}=\left(\Omega_{f,G}\right)_{\overline{0}}\oplus \left(\Omega_{f,G}\right)_{\overline{1}}$, 
$\overline i\in\ZZ/2\ZZ$, is given by 
\begin{subequations}
\begin{equation}\label{equation: omega polarization}
\left(\Omega_{f,G}\right)_{\overline{0}}:=\bigoplus_{\substack{g\in G\\ N-N_g\equiv 0\ (\text{\rm mod } 2)}}\Omega_{f,g},
\quad
\left(\Omega_{f,G}\right)_{\overline{1}}:=\bigoplus_{\substack{g\in G\\ N-N_g\equiv 1\ (\text{\rm mod } 2)}}\Omega_{f,g},
\end{equation}
\begin{equation}
\Omega_{f,g}:=\left(\Omega_{f^g}\right)^G.
\end{equation}
\end{subequations}
\end{definition}
\begin{definition}\label{def_JfG}
Define a non-degenerate $\CC$-bilinear form 
$J_{f,G}:\Omega'_{f,G}\otimes_{\CC} \Omega'_{f,G}\longrightarrow \CC$, called the {\em orbifold residue pairing}, by
\begin{subequations}
\begin{equation}\label{equation: JfG}
J_{f,G}:=\bigoplus_{g\in G}J_{f,g},
\end{equation}
where $J_{f,g}$ is a perfect $\CC$-bilinear form 
$J_{f,g}:\Omega'_{f,g}\otimes_{\CC} \Omega'_{f,{g^{-1}}}\longrightarrow \CC$ defined by 
\begin{equation}
J_{f,g}\left(\omega_1,\omega_2\right):=(-1)^{N-N_g}\cdot \epi\left[-\frac{1}{2}\age(g)\right]\cdot |G|\cdot {\rm Res}_{\Fix(g)}\left[
\begin{gathered}
\phi \psi dx_{i_1}\wedge\dots\wedge dx_{i_{N_g}}\\
\frac{\partial f^g}{\partial x_{i_1}}\dots\frac{\partial f^g}{\partial x_{i_{N_g}}}
\end{gathered}
\right]
\end{equation}
for $\omega_1=[\phi dx_{i_1}\wedge\dots\wedge dx_{i_{N_g}}]\in \Omega'_{f,g}$ and 
$\omega_2=[\psi dx_{i_1}\wedge\dots\wedge dx_{i_{N_g}}]\in \Omega'_{f,g^{-1}}$, 
where $x_{i_1},\dots, x_{i_{N_g}}$ are coordinates of $\Fix(g)$.
For each $g\in G$ with $\Fix(g)=\{{\bf 0}\}$, we define 
\begin{equation}
J_{f,g}\left(1_g,1_{g^{-1}}\right):=(-1)^{N}\cdot \epi\left[-\frac{1}{2}\age(g)\right]\cdot |G|,
\end{equation}
where $1_g\in\Omega'_{f,g}$ and $1_{g^{-1}}\in\Omega'_{f,g^{-1}}$ denote the constant functions on $\{{\bf 0}\}$ 
whose values are $1$. 
\end{subequations}
\end{definition}
\begin{proposition}\label{prop:G-sym J}
The $\CC$-bilinear form $J_{f,G}$ is $G$-twisted $\ZZ/2\ZZ$-graded symmetric in the sense that 
\begin{equation}
J_{f,G}(\omega_1,\omega_2)=(-1)^{N-N_g}\cdot \epi\left[-\age(g)\right]\cdot J_{f,G}(\omega_2,\omega_1)
\end{equation}
for $\omega_1\in \Omega'_{f,g}$ and $\omega_2\in \Omega'_{f,g^{-1}}$.
\end{proposition}
\begin{proof}
Let the notations be as in Definition~\ref{def_JfG}. Since $\Fix(g)=\Fix(g^{-1})$, $f^g=f^{g^{-1}}$ and $\age(g)+\age(g^{-1})=N-N_g$,
we have 
\begin{align*}
J_{f,G}(\omega_1,\omega_2)&=J_{f,g}\left(\omega_1,\omega_2\right)\\
&=\epi\left[-\frac{1}{2}\age(g)+\frac{1}{2}\age(g^{-1})\right]\cdot J_{f,g^{-1}}\left(\omega_2,\omega_1\right)\\
&=(-1)^{N-N_g}\cdot \epi\left[-\age(g)\right]\cdot J_{f,G}\left(\omega_2,\omega_1\right).
\end{align*}
\end{proof}

For a $\CC$-algebra $R$, denote by $\Aut_{\CC\text{-}{\rm alg}}(R)$ the group of all $\CC$-algebra automorphisms of $R$. 
Note that $G$ is naturally identified with a subgroup of $\Aut_{\CC\text{-}{\rm alg}}(\CC[x_1,\dots, x_N])$.
\begin{definition}\label{def_Aut}
Define the group $\Aut(f,G)$ of automorphisms of $(f,G)$ as 
\[
\Aut(f,G):=\{\varphi\in \Aut_{\CC\text{-}{\rm alg}}(\CC[x_1,\dots, x_N])~|~\varphi(f)=f,\ \varphi\circ g\circ \varphi^{-1}\in G 
\text{ for all }g\in G\}.
\]
\end{definition}
It is obvious that $G$ is naturally identified with a subgroup of $\Aut(f,G)$.
Note that a $\CC$-algebra automorphism $\varphi\in \Aut_{\CC\text{-}{\rm alg}}(\CC[x_1,\dots, x_N])$ is $G$-equivariant if and only if 
$\varphi\circ g\circ \varphi^{-1}=g$ for all $g\in G$.
\begin{remark}\label{rem:skew}
Let $\CC[x_1,\dots, x_N]*G$ be the skew group ring which is a $\CC$-vector space $\CC[x_1,\dots, x_N]\otimes_\CC \CC[G]$ 
with a product defined as $(\phi_1\otimes g_1)(\phi_2\otimes g_2)=(\phi_1g_1(\phi_2))\otimes g_1g_2$ for any $\phi_1,\phi_2\in \CC[x_1,\dots, x_N]$ and $g_1,g_2\in G$.
Then the group $\Aut(f,G)$ can be regarded as the subgroup of all $\varphi'\in \Aut_{\CC\text{-}{\rm alg}}(\CC[x_1,\dots, x_N]*G)$ 
such that $\varphi'(f\otimes\id)=f\otimes\id$. 
For $\varphi\in\Aut(f,G)$, the correspondence element in $\Aut_{\CC\text{-}{\rm alg}}(\CC[x_1,\dots, x_N]*G)$ 
is given by $\phi\otimes g\mapsto \varphi(\phi)\otimes (\varphi\circ g\circ \varphi^{-1})$.
\end{remark}
An element $\varphi\in\Aut(f,G)$ regarded as a bi-regular map on $\CC^N$ maps 
$\Fix(\varphi\circ g\circ \varphi^{-1})$ to $\Fix(g)$ for each $g\in G$.
Hence, the group $\Aut(f,G)$ acts naturally on $\Omega'_{f,G}$ by 
\begin{equation}\label{conj-action}
\Omega'_{f,g}\longrightarrow \Omega'_{f,\varphi\circ g\circ \varphi^{-1}},\quad \omega\mapsto \varphi^*|_{\Fix(g)}\omega,
\end{equation}
where $\varphi^*|_{\Fix(g)}$ denotes the restriction of the pullback $\varphi^*$ of differential forms to $\Fix(g)$.
In order to simplify the notation, for each $\varphi\in\Aut(f,G)$, we shall denote by $\varphi^\ast$ the action of $\varphi$ on $\Omega'_{f,G}$. 
It also follows that $\Aut(f,G)$ acts naturally on $\Omega_{f,G}$.
\subsection{Axioms}
\begin{definition}\label{axioms}
A {\em $G$-twisted Jacobian algebra} of $f$ is a $\ZZ/2\ZZ$-graded $\CC$-algebra $\Jac'(f, G)=\Jac'(f,G)_{\overline{0}}\oplus \Jac'(f,G)_{\overline{1}}$, $\overline i\in\ZZ/2\ZZ$, satisfying the following axioms:
\begin{enumerate}
\item \label{axiom_vs}
For each $g\in G$, there is a $\CC$-module $\Jac'(f,g)$ isomorphic to $\Omega'_{f,g}$ as a $\CC$-module 
satisfying the following conditions: 
\begin{enumerate}
\item
For the identity $\id$ of $G$,
\begin{equation}\label{axiom_Jacid}
\Jac'(f,\id)=\Jac(f).
\end{equation}
\item We have 
\begin{subequations}
\begin{equation}
\Jac'(f,G)_{\overline{0}}=\bigoplus_{\substack{g\in G\\ N-N_g\equiv 0\ (\text{\rm mod } 2)}}\Jac'(f,g),
\end{equation}
\begin{equation}
\Jac'(f,G)_{\overline{1}}=\bigoplus_{\substack{g\in G\\ N-N_g\equiv 1\ (\text{\rm mod } 2)}}\Jac'(f,g).
\end{equation}
\end{subequations}
\end{enumerate}
\item \label{axiom_algebra}
The $\ZZ/2\ZZ$-graded $\CC$-algebra structure $\circ$ 
on $\Jac'(f,G)$ satisfies
\begin{equation}\label{equation: multiplication g,h to gh}
\Jac'(f,g) \circ \Jac'(f,h)\subset \Jac'(f,gh),\quad g,h\in G,
\end{equation}
and the $\CC$-subalgebra $\Jac'(f,\id)$ of $\Jac'(f,G)$ coincides with the $\CC$-algebra $\Jac(f)$.
\item \label{axiom_=JacOmega}
The $\ZZ/2\ZZ$-graded $\CC$-algebra $\Jac'(f,G)$ is such that
the $\CC$-module $\Omega'_{f,G}$ has a structure of a $\Jac'(f,G)$-module
\begin{equation}
\vdash: \Jac'(f,G)\otimes\Omega'_{f,G}\longrightarrow \Omega'_{f,G},\quad X\otimes \omega\mapsto X\vdash \omega,
\end{equation}
satisfying the following conditions:
\begin{enumerate}
\item \label{axiom_idmodule}
For any $g,h\in G$ we have
\begin{equation}
\Jac'(f,g)\vdash \Omega'_{f,h}\subset \Omega'_{f,gh},
\end{equation}
and the $\Jac'(f,\id)$--module structure on $\Omega'_{f,g}$ coincides with 
the $\Jac(f)$-module structure on $\Omega_{f^g}$ given by Corollary~\ref{corollary: Jac module on omega fg}.
\item \label{axiom_isom}
By choosing a nowhere vanishing $N$-form, we have the following isomorphism
\begin{equation}\label{eq:isom2}
\Jac'(f,G)\stackrel{\cong}{\longrightarrow }\Omega'_{f,G},\quad X\mapsto X\vdash\zeta,
\end{equation}
where $\zeta$ is the residue class in $\Omega'_{f,\id}=\Omega_f$ of the $N$-form.
Namely, $\Omega'_{f,G}$ is a free $\Jac'(f,G)$-module of rank one.
\end{enumerate}
\item 
There is an induce action of $\Aut(f,G)$ on $\Jac'(f,G)$ given by 
\begin{equation}
\varphi^\ast(X) \vdash \varphi^\ast(\zeta) := \varphi^\ast(X \vdash \zeta),\quad \varphi \in \Aut(f,G),\ X\in \Jac'(f,G),
\end{equation}
where $\zeta$ is an element in $\Omega'_{f,\id}$ giving the isomorphism in axiom~\eqref{axiom_isom}.
The algebra structure of $\Jac'(f,G)$ satisfies the following conditions: 
\begin{enumerate}
\item\label{axiom_ringauto}
It is $\Aut(f,G)$-invariant, namely, 
\begin{equation}
\varphi^\ast(X)\circ \varphi^\ast(Y)=\varphi^\ast (X\circ Y),\quad\varphi \in \Aut(f,G),\ X,Y\in \Jac'(f,G).
\end{equation}
\item\label{axiom: G-twisted comm}
It is $G$-twisted $\ZZ/2\ZZ$-graded commutative, namely, for any $g,h\in G$ and $X\in \Jac'(f,g)$, $Y\in\Jac'(f,h)$, we have  
\begin{equation}
X\circ Y=(-1)^{\overline{X}\cdot\overline{Y}} g^\ast (Y)\circ X,
\end{equation}
where $\overline{X}=N-N_g$ and $\overline{Y}=N-N_h$ are the $\ZZ/2\ZZ$-grading of $X$ and $Y$, 
and $g^\ast$ is the induced action of $g$ considered as an element of $\Aut(f,G)$.
\end{enumerate}
\item \label{axiom_J_fG}
For any $g,h\in G$ and $X\in \Jac'(f,g)$, $\omega\in\Omega'_{f,h}$, $\omega'\in\Omega'_{f,G}$, we have      
\begin{equation}
J_{f,G}(X\vdash\omega,\omega')=(-1)^{\overline{X}\cdot \overline{\omega}} 
J_{f,G}\left(\omega,((h^{-1})^\ast X)\vdash\omega'\right),
\end{equation}
where $\overline{X}=N-N_g$ and $\overline{\omega}=N-N_h$ are the $\ZZ/2\ZZ$-grading of $X$ and $\omega$, 
and $(h^{-1})^\ast$ is the induced action of $h^{-1}$ considered as an element of $\Aut(f,G)$.
\item\label{axiom: G-H} 
Let $G'$ be a finite subgroup of $G_f$ such that $G \subseteq G'$.
Fix a nowhere vanishing $N$-form and denote by $\zeta$ its residue class in $\Omega'_{f,\id}$. 
By axiom~\eqref{axiom_isom} for $G,G'$, fix the isomorphisms given by $\zeta$;
\begin{eqnarray}
& &\Jac'(f,G)\stackrel{\cong}{\longrightarrow }\Omega'_{f,G},\quad X\mapsto X\vdash\zeta,\\
& &\Jac'(f,G')\stackrel{\cong}{\longrightarrow }\Omega'_{f,G'},\quad X'\mapsto X'\vdash\zeta.
\end{eqnarray}
Then, the injective map $\Omega'_{f,G}\longrightarrow \Omega'_{f,G'}$ induced by
the identity maps $\Omega'_{f,g}\longrightarrow \Omega'_{f,g}$, $g\in G$ 
yields an injective map of the $\ZZ/2\ZZ$-graded $\CC$-modules $\Jac'(f,G) \to \Jac'(f,G')$, which is an algebra-homomorphism.
\end{enumerate}
\end{definition}
\subsection{Comments on the axioms}
Such a class $\zeta\in\Omega'_{f,\id}$ giving the isomorphism in axiom~\eqref{axiom_isom}
is a non-zero constant multiple of the residue class of $dx_1\wedge \dots\wedge dx_N$.
It follows that the $\Aut(f,G)$-action on $\Jac'(f,G)$ does not depend on the choice of $\zeta$. 
In particular, the $\Aut(f,G)$-action on $\Jac'(f,\id)=\Jac(f)$ is nothing but the usual one which is induced by
the natural $\Aut(f,G)$-action on $\CC[x_1,\dots, x_N]$. 
For different choices of $\zeta$ we get isomorphic algebras.
Axioms~\eqref{axiom_ringauto}, \eqref{axiom: G-twisted comm} and \eqref{axiom_J_fG} are  
naturally expected by keeping the skew group ring $\CC[x_1,\dots, x_N]*G$ in mind (see also Remark~\ref{rem:skew}).
Indeed, our axioms are motivated by some intuitive properties of the ``Jacobian algebra of $f$ in the non-commutative 
ring $\CC[x_1,\dots, x_N]*G$". 
Axiom~\eqref{axiom: G-twisted comm} can also be found in \cite{K1}, while the others seem to be new.
We have not used the commutativity of $G$ in the axioms in Definition~\ref{axioms} except for the last one~\eqref{axiom: G-H}.
Instead of $G_f$ there, by the use of the largest subgroup $G_{f,nc}$ of ${\rm GL}(N;\CC)$ respecting $f$ 
whose restriction $f^g$ to $\Fix(g)$ gives a finite dimensional Jacobian algebra $\Jac(f^g)$ for all $g\in G_{f,nc}$, 
the definition can naturally be extended to the non-abelian case, namely, the case when $G$ is a subgroup of $G_{f,nc}$.  
\subsection{Conjecture and the definition}
We shall denote by $v_\id$ the residue class of $1\in\CC[x_1,\dots, x_N]$ in $\Jac'(f,\id)=\Jac(f)$, which is the unit with respect to 
the product structure $\circ$ since by axiom~\eqref{axiom_J_fG} we have
\begin{equation}
J_{f,G}((X\circ v_\id)\vdash\zeta,\omega)=J_{f,G}(X\vdash(v_\id\vdash\zeta),\omega)=J_{f,G}(X\vdash \zeta,\omega)
\end{equation}
for all $X\in \Jac'(f,G)$, $\omega\in\Omega'_{f,G}$ and $\zeta\in \Omega_{f,\id}$ giving the isomorphism~\eqref{eq:isom2}.
Note also that $\varphi^\ast(v_\id)=v_\id$ for all $\varphi\in\Aut(f,G)$ since 
$\varphi^\ast(v_\id) \vdash \varphi^\ast(\zeta) = \varphi^\ast(v_\id \vdash \zeta)=\varphi^\ast(\zeta)=v_\id\vdash\varphi^\ast(\zeta)$. 
In particular, $v_\id$ is $G$-invariant.
By the isomorphism~\eqref{eq:isom2}, it follows from \eqref{conj-action} that 
\begin{equation}
\varphi^\ast(\Jac'(f,g))=\Jac'(f,\varphi\circ g \circ \varphi^{-1}),\quad \varphi\in\Aut(f,G). 
\end{equation}
In particular, $g^\ast(\Jac'(f,h)) = \Jac'(f,ghg^{-1})$ for $g, h\in G$.
Now, $G$ is a commutative group, we have $g^\ast(\Jac'(f,h)) =\Jac'(f,h)$. 
Since the product structure $\circ$ is also $G$-invariant by axiom~\eqref{axiom_ringauto}
it follows that the $G$-invariant subspace of ${\rm Jac}'(f,G)$ 
has a structure of a $\ZZ/2\ZZ$-graded commutative algebra, which is 
$\ZZ/2\ZZ$-graded commutative due to axiom~\eqref{axiom: G-twisted comm}.
A priori there might not be a unique $\ZZ/2\ZZ$-graded $\CC$-algebra satisfying the axioms in Definition~\ref{axioms},
nevertheless we expect the following
\begin{conjecture}\label{existence&uniqueness}
Let the notations be as above.
\begin{enumerate}
\item[(a)]
A $G$-twisted Jacobian algebra ${\rm Jac}'(f,G)$ of $f$ should exist. 
\item[(b)]
The subalgebra $\left(\Jac'(f,G)\right)^G$ should be 
uniquely determined by $(f,G)$ up to isomorphism.
\end{enumerate}
\end{conjecture}
\begin{definition}
Suppose that Conjecture~\ref{existence&uniqueness} holds for the pair $(f,G)$. 
The $\ZZ/2\ZZ$-graded commutative algebra 
\begin{equation}
\Jac(f,G):=\left(\Jac'(f,G)\right)^G
\end{equation}
is called the {\em orbifold Jacobian algebra} of $(f,G)$.
\end{definition}
In Theorem~\ref{theorem_N}
we prove Conjecture \ref{existence&uniqueness} (actually a stronger statement than it),
for a large class of polynomials $f$ --- so--called invertible polynomials and any symmetry group $G$ of it.
Under the isomorphism in axiom~\eqref{axiom_isom}, it follows from axiom~\eqref{axiom_J_fG} that 
the non-degenerate $G$-twisted $\ZZ/2\ZZ$-graded symmetric $\CC$-bilinear form $J_{f,G}$ on $\Omega'_{f,G}$ equips $\Jac'(f,G)$ with the structure of $\ZZ/2\ZZ$-graded Frobenius algebra.
If $G$ is a subgroup of $G_f^\SL$, then $\age(g)\in\ZZ$ for all $g\in G$,
the residue class $\zeta$ is $G$-invariant and 
the pairing $J_{f,G}$ induces a $\ZZ/2\ZZ$-graded symmetric pairing on $\Omega_{f,G}$ 
due to the $G$-twisted $\ZZ/2\ZZ$-graded commutativity (Proposition~\ref{prop:G-sym J}).
Therefore, it follows easily that ${\rm Jac}(f,G)$ for $G\subseteq G_f^\SL$ is equipped with a structure of 
$\ZZ/2\ZZ$-graded commutative Frobenius algebra, which will be of our main interest. 

\section{Orbifold Jacobian algebras for invertible polynomials}
\subsection{Invertible polynomials}
A polynomial $f\in\CC[x_1,\dots, x_N]$ is called a {\em weighted homogeneous} polynomial  
if there are positive integers $w_1,\dots ,w_N$ and $d$ such that 
\begin{equation}
f(\lambda^{w_1} x_1, \dots, \lambda^{w_N} x_N) = \lambda^d f(x_1,\dots ,x_N)
\end{equation}
for all $\lambda \in \CC^\ast$.
We call $(w_1,\dots ,w_N;d)$ a system of weights of $f$.
A weighted homogeneous polynomial $f$ is called {\em non-degenerate}
if it has at most an isolated critical point at the origin in $\CC^N$, equivalently, if the Jacobian algebra $\Jac(f)$ of $f$ 
is finite-dimensional.
\begin{definition}
A weighted homogeneous polynomial $f\in\CC[x_1,\dots, x_N]$ is called {\em invertible} if 
the following conditions are satisfied.
\begin{enumerate}
\item 
The number of variables ($=N$) coincides with the number of monomials in the polynomial $f$, namely, 
\begin{equation}
f(x_1,\dots ,x_N)=\sum_{i=1}^N c_i\prod_{j=1}^Nx_j^{E_{ij}}
\end{equation}
for some coefficients $c_i\in\CC^\ast$ and non-negative integers 
$E_{ij}$ for $i,j=1,\dots, N$.
\item
The matrix $E:=(E_{ij})$ is invertible over $\QQ$.
\item
The polynomial $f$ and the {\em Berglund--H\"{u}bsch transpose} $\widetilde f$ of $f$ defined by
\begin{equation}
\widetilde{f}(x_1,\dots ,x_N):=\sum_{i=1}^N c_i\prod_{j=1}^N x_j^{E_{ji}}
\end{equation}
are non-degenerate.
\end{enumerate}
\end{definition}
\begin{definition}\label{def: q-weights}
Let $f(x_1,\ldots ,x_N)=\sum_{i=1}^N c_i\prod_{j=1}^N x_j^{E_{ij}}$ be an invertible polynomial.
Define rational numbers $q_1,\ldots, q_N$ by the unique solution of the equation
\begin{equation}
E
\begin{pmatrix}
q_1\\
\vdots\\
q_N
\end{pmatrix}
=
\begin{pmatrix}
1\\
\vdots\\
1
\end{pmatrix}.
\end{equation}
Namely, set $q_i := w_i/d$, $i=1, \ldots , N$, for the system of weights $(w_1,\ldots ,w_N;d)$.
\end{definition}

If $f(x_1, \ldots, x_N)$ is an invertible polynomial, then we have 
\begin{equation}\label{eq_Gfinvertible}
G_f=\left\{(\lambda_1,\dots ,\lambda_N)\in (\CC^\ast)^N \, \left| \, \prod_{j=1}^N \lambda_j ^{E_{1j}}=\dots =\prod_{j=1}^N\lambda_j^{E_{Nj}}=1 \right\} \right. ,
\end{equation}
and hence $G_f$ is a finite group. 
It is easy to see that $G_f$ contains an element $g_0:=(q_1,\dots, q_N)$.

It is important to note the following
\begin{proposition}\label{prop_GfSL_proper}
The group $G_f^\SL=G_f\cap \SL(N;\CC)$ is a proper subgroup of $G_f$.
\end{proposition}
\begin{proof}
Let $\widetilde f$ be the Berglund--H\"ubsch transpose of $f$. 
It is known by \cite{ET11} and \cite{Kr} (see also Proposition~2 in \cite{EG-ZT}) that 
\[
G_f^\SL\cong \Hom(G_{\widetilde f}/\langle ({\widetilde q}_1,\dots, {\widetilde q}_N) \rangle,\CC^\ast)\subsetneq
\Hom(G_{\widetilde f},\CC^\ast)\cong G_f,
\]
where $({\widetilde q}_1,\dots, {\widetilde q}_N)$ is the unique solution of the equation
$({\widetilde q}_1,\dots, {\widetilde q}_N)E=(1,\dots, 1)$.
\end{proof}

The following is our first theorem of this paper.
\begin{theorem} \label{theorem_N}
Let $f$ be an invertible polynomial and $G$ a subgroup of $G_f$.
There exists a unique $G$-twisted Jacobian algebra $\Jac'(f,G)$ of $f$ up to isomorphism.
Namely, it is uniquely characterized by the axioms in Definition~\ref{axioms}.
In particular, the orbifold Jacobian algebra $\Jac(f,G)$ of $(f,G)$ exists. 
\end{theorem}
In the subsequent subsections, we first prepare some notation, and then prove the uniqueness, 
and finally prove the existence.
\subsection{Notations}

Let $f(x_1,\ldots ,x_N)=\sum_{i=1}^N c_i\prod_{j=1}^N x_j^{E_{ij}}$ be an invertible polynomial.
Without loss of generality one may assume that $c_i=1$ for $i=1,\ldots, N$ by rescaling the variables.
According to \cite{KS}, an invertible polynomial $f$ can be written as a Thom--Sebastiani sum
$f=f_1\oplus\cdots\oplus f_p$ of invertible polynomials (in groups of different variables) 
$f_\nu$, $\nu=1,\dots, p$ of the following types:
\begin{enumerate}
\item $x_1^{a_1}x_2+x_2^{a_2}x_3+\dots+x_{m-1}^{a_{m-1}}x_m+x_m^{a_m}$ (chain type; $m \geq 1$);
\item  $x_1^{a_1}x_2+x_2^{a_2}x_3+\dots+x_{m-1}^{a_{m-1}}x_m+x_m^{a_m}x_1$ (loop type; $m \geq 2$).
\end{enumerate}
\begin{remark}
In \cite{KS} the authors distinguished also polynomials of the so called Fermat type: $x_1^{a_1}$, 
which is regarded as a chain type polynomial with $m = 1$ in this paper.
\end{remark}

We shall use the monomial basis of the Jacobian algebra $\Jac(f_\nu)$
\begin{proposition}[cf. \cite{Kreuzer}]\label{monomial basis}
For an invertible polynomial $f_\nu=x_1^{a_1}x_2+x_2^{a_2}x_3+\dots+x_{m-1}^{a_{m-1}}x_m+x_m^{a_m}$ of chain type with $m\ge 1$, 
the Jacobian algebra $\Jac(f_\nu)$ has a monomial basis consisting of all the monomials $x_1^{k_1}\cdots x_m^{k_m}$ such that
\begin{enumerate}
\item[1)] $0\le k_i\le a_i-1$,
\item[2)] if
\[
k_i=
\begin{cases}
a_i-1 {\text{ for all odd $i$, $i\le 2s-1$,}}\\
0 {\text{ for all even $i$, $i\le 2s-1$,}}
\end{cases}
\]
then $k_{2s}=0$.
\end{enumerate}

For an invertible polynomial $f_\nu=x_1^{a_1}x_2+x_2^{a_2}x_3+\dots+x_{m-1}^{a_{m-1}}x_m+x_m^{a_m}x_1$ of loop type with $m\ge 2$, 
the Jacobian algebra $\Jac(f_\nu)$ has a monomial basis consisting of all the monomials $x_1^{k_1}\cdots x_m^{k_m}$ with
$0 \le k_i \le a_i-1$.
\end{proposition}
Let $I_g := \{i_1,\dots,i_{N_g}\}$ be a subset of $\{1,\dots, N\}$ such that $\Fix(g)=\{x\in\CC^N~|~x_{j}=0, j\notin I_g\}$.
In particular, $I_{\id}=\{1,\dots, N\}$. Denote by $I_g^c$ the complement of $I_g$ in $I_\id$.
In what follows, we are mostly interested in special pairs of elements of $G_f$.
\begin{definition}\label{definition: primitive decomposition}
Let $f=f(x_1,\dots,x_N)$ be an invertible polynomial.
\begin{enumerate}
\item
An ordered pair $(g,h)$ of elements of $G_f$ is called {\em spanning} if 
\begin{equation}
I_g\cup I_h \cup I_{gh} = \{1,\dots,N\}.
\end{equation}
\item
For a spanning pair $(g,h)$ of elements of $G_f$, define $I_{g,h}:=I_g^c \cap I_h^c$.
\item
For a spanning pair $(g,h)$ of elements of $G_f$, 
there always exist $g_1,g_2,h_1,h_2 \in G_f$ such that $g = g_1g_2$ and $h = h_1h_2$ 
with $g_2 h_2 = \id$ and $I_{g_1,h_1} = \emptyset$.
The tuple $(g_1, g_2, h_1 ,h_2)$ is called the {\em factorization} of $(g,h)$.
\end{enumerate}
\end{definition}
\begin{remark}
For a spanning pair $(g,h)$ of elements of $G_f$, up to a reordering of the variables, we have
\begin{equation}\label{eq: gh explicit form}
\begin{aligned}
g = &(0,\dots,0, \alpha_1, \dots, \alpha_p, \beta_1, \dots, \beta_q)\\ 
h = &(\gamma_1,\dots, \gamma_r, 0,\dots,0, 1-\beta_1,\dots, 1-\beta_q),
\end{aligned}
\end{equation}
for some rational numbers $0< \alpha_i, \beta_i,\gamma_i <1$ 
and integers $p,q,r$ such that $0 \le r \le N_{g}$ and $N_{g} + p + q = r + N_{h} + q = N$. 
In this presentation,  
we have $I_g \cap I_h = \{i_{r+1},\dots, i_{N-q-p}\}$, $I_{g,h}= \{i_{N-q+1},\dots, i_{N}\}$ and 
\begin{align*}
g_1 = &(0,\dots,0, \alpha_1, \dots, \alpha_p, 0, \dots, 0),\\ 
g_2 = &(0,\dots,0, 0, \dots, 0, \beta_1, \dots, \beta_q),\\ 
h_1 = &(\gamma_1,\dots, \gamma_r, 0,\dots,0, 0,\dots, 0),\\
h_2 = &(0,\dots,0, 0,\dots,0, 1-\beta_1,\dots, 1-\beta_q).
\end{align*}
\end{remark}

We introduce one of most important objects in this paper.
\begin{definition}\label{definition: Hgh}
Let $f=f(x_1,\dots,x_N)$ be an invertible polynomial.
For each spanning pair $(g,h)$ of elements of $G_f$, 
define a polynomial 
$H_{g,h}\in\CC[x_1,\dots, x_N]$ by
\begin{equation}
H_{g,h}:=
\begin{cases}
\widetilde{m}_{g,h}\det \left(\frac{\partial^2 f}{\partial x_{i} \partial x_{j}}\right)_{i,j\in I_{g,h}} 
& \text{if} \quad I_{g,h} \neq \emptyset\\
1 & \text{if} \quad I_{g,h} = \emptyset
\end{cases},
\end{equation}
where $\widetilde{m}_{g,h}\in\CC^\ast$ is the constant uniquely determined by the following equation in $\Jac(f^{gh})$
\begin{equation}\label{H and hessians}
\frac{1}{\mu_{f^{g\cap h}}}[\hess(f^{g\cap h})H_{g,h}]=\frac{1}{\mu_{f^{gh}}}[\hess(f^{gh})],
\end{equation}
where $f^{g\cap h}$ is an invertible polynomial given by the restriction $f|_{\Fix(g)\cap\Fix(h)}$ of $f$ to 
the locus $\Fix(g)\cap \Fix(h)$.
\end{definition}
\begin{remark}
The polynomial $H_{g,h}$ is a non-zero constant multiple of the determinant of a minor of the Hessian matrix of $f(x_1,\dots,x_N)$. 
Since $I_g\cap I_h\subseteq I_{gh}$ and $I_{g,h}\subseteq I_{gh}$, $\hess(f^{g\cap h})$ and $H_{g,h}$ define elements of $\Jac(f^{gh})$. 
\end{remark}
\begin{remark}
Let $(g,h)$ be a spanning pair of elements of $G_f$.
Suppose that $\Fix(g)=\{{\bf 0\}}$. Then $h=g^{-1}$. It is easy to check that $H_{g,h}=\frac{1}{\mu_f}[\hess(f)]$ 
by the explanation of $\widetilde{m}_{g,h}$ below.
Recall also Definition~\ref{def:1} that if $\Fix(g)\cap \Fix(h)=\{{\bf 0\}}$ 
then $\mu_{f^{g\cap h}}=1$ and $\hess(f^{g\cap h})=1$.
\end{remark}
We explain the existence and the uniqueness of $\widetilde{m}_{g,h}$ in Definition~\ref{definition: Hgh}. Suppose that $f=f_1\oplus\cdots\oplus f_p$ is a Thom--Sebastiani sum 
such that each $f_\nu =f_\nu(x_{i_1},\dots,x_{i_m})$, $\nu=1,\dots, p$ is either of chain type or loop type. 
Set $I_\nu := \{i_1,\dots,i_m\} \subseteq \{1,\dots,N\}$ for each $\nu$.
Then $\Jac(f) = \Jac(f_1) \otimes \dots \otimes \Jac(f_p)$ and 
\begin{equation}
  \det \left(\frac{\partial^2 f}{\partial x_{i} \partial x_{j}}\right)_{i,j \in I_\id} = \ \prod_{\nu=1}^p \det \left(\frac{\partial^2 f_\nu}{\partial x_{i} \partial x_{j}}\right)_{i,j \in I_\nu}.
\end{equation}
For each $g\in G_f$ and $f_\nu$ as above the following holds:
\begin{itemize}
\item If $f_\nu$ is of the chain type $f_\nu = x_{i_1}^{a_1} x_{i_2} + \dots + x_{i_{m-1}}^{a_{m-1}}x_{i_m} + x_{i_m}^{a_m}$, 
then there exists $l$, $0\le l\le m$ such that $\{i_1,\dots,i_{l}\} \subseteq I_{g}^c$ and $\{i_{l+1},\dots,i_m\} \subseteq I_{g}$. 
\item If $f_\nu$ is of loop type, then $I_\nu \subseteq I_{g}$ or $I_\nu \subseteq I_{g}^c$.
\end{itemize}
We classify the possible cases of $I_{g,h}$. 
\begin{lemma}\label{lemma_cases}
Let $(g,h)$ be a spanning pair of elements of $G_f$.
Suppose that $f=f_1\oplus\cdots\oplus f_p$ is a Thom--Sebastiani sum 
such that each $f_\nu$, $\nu=1,\dots, p$ is either of chain type or loop type. 
Then, for each $f_\nu = f_\nu(x_{i_1}, \dots, x_{i_m})$, the either one of the following holds:
\begin{enumerate}
\item[(i)]
$f_\nu$ is of chain type and, for some $0 \leq l \leq m$,
\begin{itemize}
\item[(a)]
$\{i_1,\dots, i_m\}\subseteq I_g$, $\{i_1,\dots, i_l\}\subseteq I_h^c$ and $\{i_{l+1},\dots, i_m\}\subseteq I_h$,
\item[(a')] 
$\{i_1,\dots, i_m\}\subseteq I_h$, $\{i_1,\dots, i_l\}\subseteq I_g^c$ and $\{i_{l+1},\dots, i_m\}\subseteq I_g$,
\item[(b)] 
$\{i_1,\dots, i_l\}\subseteq I_{g,h}$ and $\{i_{l+1},\dots, i_m\}\subseteq I_g\cap I_h$.
\end{itemize}
\item[(ii)]
$f_\nu$ is of loop type and 
\begin{itemize}
\item[(a)]
$\{i_1,\dots,i_m\}\subseteq I_g\cap I_h$,
\item[(b)]
$\{i_1,\dots,i_m\}\subseteq I_g\cap I_h^c$,
\item[(b')]
$\{i_1,\dots,i_m\}\subseteq  I_g^c\cap I_h$,
\item[(c)]
$\{i_1,\dots,i_m\}\subseteq I_{g,h}$.
\end{itemize}
\end{enumerate}
\end{lemma}
\begin{proof}
It is straightforward from the explicit form of an invertible polynomial of each type and the group action on it. 
\end{proof}
Obviously, only polynomials $f_\nu$ satisfying $I_\nu \cap I_{g,h} \neq \emptyset$ contribute non-trivially to $H_{g,h}$.
Such an $f_\nu$ satisfies the either one of the following two by Lemma~\ref{lemma_cases}:
\begin{itemize}
\item[(a)]
$I_\nu = \{ i_1,\dots,i_m \} \subseteq I_{g,h}$.
\item[(b)]
$f_\nu$ is of the chain type and, for some $0 \leq l \leq m-1$,
$\{i_1,\dots,i_{l}\} \subseteq I_{g,h}$ and $\{i_{l+1},\dots,i_m\} \subseteq I_g \cap I_h$.
\end{itemize}
Set $\Gamma_a:=\{\nu~|~f_\nu \text{ satisfies (a)}\}$ and $\Gamma_b:=\{\nu~|~f_\nu \text{ satisfies (b)}\}$.
Since $I_{gh}=I_{g,h}\cup(I_g\cap I_h)$, we have 
\begin{equation}
f^{gh} = \bigoplus_{\nu_a\in\Gamma_a} f_{\nu_a} \oplus \bigoplus_{\nu_b\in\Gamma_b} f_{\nu_b}
\oplus\bigoplus_{\substack{\nu\text{ s.t. }\\ I_\nu\subseteq I_g\cap I_h}} f_{\nu},
\end{equation}
where $\oplus$ denotes a Thom-Sebastiani sum and hence
\begin{equation}
\Jac(f^{gh}) = \bigotimes_{\nu_a\in\Gamma_a} \Jac(f_{\nu_a}) \otimes \bigotimes_{\nu_b\in\Gamma_b} \Jac(f_{\nu_b})
\otimes\bigotimes_{\substack{\nu\text{ s.t. }\\ I_\nu\subseteq I_g\cap I_h}} \Jac(f_{\nu}).
\end{equation}
Consider the factorization
\begin{equation}
\det \left(\frac{\partial^2 f}{\partial x_{i} \partial x_{j}}\right)_{i,j\in I_{g,h}} 
=\prod_{\nu_a \in \Gamma_a} \widetilde{H}_{a}^{(\nu_a)}\cdot  \prod_{\nu_b \in \Gamma_b} \widetilde{H}_{b}^{(\nu_b)},
\end{equation}
where
\begin{equation}
\widetilde{H}_{a}^{(\nu_a)}:=\det \left(\frac{\partial^2 f_{\nu_a}}{\partial x_{i} \partial x_{j}}\right)_{i,j \in I_{\nu_a}}, \quad
\widetilde{H}_{b}^{(\nu_b)}:=\det \left(\frac{\partial^2 f_{\nu_b}}{\partial x_{i} \partial x_{j}}\right)_{i,j \in I_{\nu_b}\cap I_{g,h}}.
\end{equation}

Suppose for simplicity that 
$f_{\nu_b}=x_{1}^{a_1} x_{2} + \dots + x_{m-1}^{a_{m-1}}x_{m} + x_{m}^{a_m}$ with $I_{\nu_b}\cap I_{g,h}=\{1,\dots, l\}$.
By a direct calculation, we have the following equalities in $\Jac(f_{\nu_b})$;
\begin{subequations}\label{eq:4.6}
\begin{equation}
\left[\widetilde{H}_{b}^{(\nu_b)}\right]=\left(\prod_{i=1}^l a_i\right)\cdot \left(\sum_{j=1}^l (-1)^{l-j}\prod_{i=1}^j a_i\right)
\left[x_{1}^{a_1-2} x_{2}^{a_2-1} \cdots x_{l}^{a_l-1} x_{l+1}\right],
\end{equation}
\begin{equation}
\left[\hess(f_{\nu_b}|_{\Fix(g)\cap \Fix(h)})\right]=\left(\prod_{i=l+1}^m a_i\right)\cdot \left(\sum_{j=l}^m (-1)^{m-j}\prod_{i=l+1}^j a_i\right)
\left[x_{l+1}^{a_{l+1}-2} x_{l+2}^{a_{l+2}-1} \cdots x_m^{a_m-1}\right],
\end{equation}
\begin{equation}
\left[\hess(f_{\nu_b})\right]=\left(\prod_{i=1}^m a_i\right)\cdot \left(\sum_{j=0}^m (-1)^{m-j}\prod_{i=1}^j a_i\right)
\left[x_{1}^{a_1-2} x_{2}^{a_2-1} \cdots x_m^{a_m-1}\right].
\end{equation}
\end{subequations}
Note that 
\begin{equation}
\mu_{f_{\nu_b}}=\sum_{j=0}^m (-1)^{m-j}\prod_{i=1}^j a_i,\quad
\mu_{f_{\nu_b}|_{\Fix(g)\cap \Fix(h)}}=\sum_{j=l}^m (-1)^{m-j}\prod_{i=l+1}^j a_i.
\end{equation}
Hence, it is straightforward to see the existence and the uniqueness of $\widetilde{m}_{g,h}$.

\begin{proposition}\label{proposition: H_gh is a jac element}
Let $f=f(x_1,\dots,x_N)$ be an invertible polynomial.
For each spanning pair $(g,h)$ of elements of $G_f$, the following holds:
\begin{itemize}
\item[(i)] The class of $H_{g,h}$ is non-zero in $\Jac(f^{gh})$.
\item[(ii)] If $I_{g,h}=\emptyset$, then $[H_{g,g^{-1}}H_{h,h^{-1}}] = [H_{gh, (gh)^{-1}}]$ in $\Jac(f)$.
\item[(iii)] For any $j \in I_{g,h}$, the class of $x_j H_{g,h}$ is zero in $\Jac(f^{gh})$.
\end{itemize}
\end{proposition}
\begin{proof}
Let the notations be as above. 
We may assume that $I_{g,h}\neq\emptyset$ since if $I_{g,h}=\emptyset$ the statements are trivially true.
The part (i) is almost clear by the equation~\eqref{H and hessians} since $[\hess(f^{gh})]$ is non-zero. 
The part (ii) follows from the normalization of $H_{g,h}$ 
by the equation~\eqref{H and hessians} in view of the equations~\eqref{eq:4.6}. 

To prove part (iii), first note that there is $\nu$, $1 \le \nu \le p$ such that $j\in I_\nu$ for some $f_\nu$ 
satisfying either one of (a) or (b) above. 
Due to the factorization of $\Jac(f^{gh})$, it is enough to show that $[x_j \widetilde{H}_a^{(\nu)}] = 0$ if $\nu \in \Gamma_a$ and $[x_j \widetilde{H}_b^{(\nu)}] = 0$ if $\nu \in \Gamma_b$.
Since the first case is almost clear, suppose that $f_\nu \in \Gamma_b$, $I_\nu=\{1,\dots, m\}$ and 
$I_\nu\cap I_{g,h}=\{1,\dots, l\}$.  
Recall again that $[\widetilde{H}_b^{(\nu)}]$ is  a non-zero constant multiple of $[x_1^{a_1-2}x_2^{a_2-1} \dots x_l^{a_l-1}x_{l+1}]$.
It is easy to calculate by induction that $[x_1^{a_1-1}x_2]=0$ and $[x_j^{a_j}x_{j+1}]=0$ in $\Jac(f_\nu)$ for $j=2,\dots, l$.
Therefore, we have $[x_j\widetilde{H}_b^{(\nu)}]=0$ in $\Jac(f_\nu)$ for $j=1,\dots, l$
(see also the description of the monomial basis in Proposition~\ref{monomial basis}). 
This completes the part (iii) of the proposition.
\end{proof}

\begin{proposition}\label{prop: supergrading agrees}
Let $f=f(x_1,\dots,x_N)$ be an invertible polynomial.
For each spanning pair $(g,h)$ of elements of $G_f$, we have 
\begin{equation}
(N-N_g) + (N-N_h) \equiv (N - N_{gh})\ (\text{\rm mod } 2).
\end{equation}
Moreover, if $I_{g,h}=\emptyset$ then $(N-N_g) + (N-N_h) = (N - N_{gh})$.
\end{proposition}
\begin{proof}
First of all, note that $N - N_g = |I_g^c|$ --- the number of elements in the set $I_g^c$.
Therefore, the following equalities yield the statement:
\begin{align*}
&N - N_{g} = |I_g^c\backslash I_{g,h}| + |I_{g,h}|, \quad N - N_{h} = |I_h^c\backslash I_{g,h}| + |I_{g,h}|,\\
&N - N_{gh} = |I_{gh}^c|=|I_g^c\backslash I_{g,h}| +  |I_h^c\backslash I_{g,h}|.
\end{align*}
\end{proof}
For each $g\in G_f$, the set $I_g\subseteq \{1,\dots, N\}$ and its complement $I_g^c$ 
will often be regarded as a subsequence of $(1,\dots, N)$:
\begin{equation}
I_g=(i_1,\dots, i_{N_g}), \ i_1<\dots <i_{N_g},
\quad I_g^c=(j_1,\dots, j_{N-N_g}),\ j_1<\dots <j_{N-N_g}.
\end{equation}
\begin{definition}\label{def: epsilon}
Let $g_1,\dots, g_k$ be elements of $G_f$ such that $I_{g_i,g_j}= \emptyset$ if $i\ne j$.
\begin{enumerate}
\item
Denote by $I_{g_1}^c \sqcup I_{g_2}^c$ the sequence given by adding the sequence $I_{g_2}^c$ at the end of 
the sequence $I_{g_1}^c$.
Define inductively $I_{g_1}^c \sqcup \dots \sqcup I_{g_k}^c$ by $\left(I_{g_1}^c \sqcup \dots \sqcup I_{g_{k-1}}^c\right)\sqcup I_{g_k}^c$.Obviously, as a set, $I_{g_1}^c \sqcup\dots\sqcup I_{g_k}^c=I_{g_1\dots g_k}^c$.
\item
Let $\sigma_{g_1,\dots, g_k}$ be the permutation which turns the sequence $I_{g_1}^c \sqcup\dots\sqcup I_{g_k}^c$ to the sequence $I_{g_1\dots g_k}^c$.
Define $\widetilde{\varepsilon}_{g_1,\dots,g_k}$ as the signature ${\rm sgn}(\sigma_{g_1,\dots, g_k})$ 
of the permutation $\sigma_{g_1,\dots, g_k}$.
\end{enumerate}
\end{definition}
It is straightforward from the definition that 
\begin{subequations}
\begin{eqnarray}
&\widetilde{\varepsilon}_{g,\id}=1=\widetilde{\varepsilon}_{\id,g},\ &g\in G_f,\\
&\widetilde{\varepsilon}_{g,h}=(-1)^{(N-N_g)(N-N_h)}\widetilde{\varepsilon}_{h,g},\ &g,h\in G_f,\ I_{g,h}=\emptyset,\\
&\widetilde{\varepsilon}_{g,g'}\widetilde{\varepsilon}_{gg',g''}=\widetilde{\varepsilon}_{g,g',g''}=
\widetilde{\varepsilon}_{g,g'g''}\widetilde{\varepsilon}_{g',g''},\ &
g,g',g''\in G_f,\ I_{g,g'}=I_{g',g''}=I_{g,g''}=\emptyset.
\end{eqnarray}
\end{subequations}
\subsection{Uniqueness}
Throughout this subsection, $f=f(x_1,\dots, x_N)$ denotes an invertible polynomial.
In this subsection, we shall show that for any $G\subseteq G_f$ the axioms in Definition~\ref{axioms} determine uniquely $\Jac'(f,G)$ up to isomorphism.
We only have to show that for $g,h\in G$ the product $\circ:\Jac'(f,g)\otimes_\CC \Jac'(f,h)\longrightarrow \Jac'(f,gh)$ 
is uniquely determined up to rescaling of generators of $\Jac(f^g)$-modules $\Jac'(f,g)$.

Take a nowhere vanishing $N$-form $dx_1\wedge \dots\wedge dx_N$
and set $\zeta:=[dx_1\wedge \dots\wedge dx_N]\in\Omega_f$.
For each subgroup $G\subseteq G_f$, fix an isomorphism in axiom~\eqref{axiom_=JacOmega}
\begin{equation}
\vdash:\Jac'(f,G)\stackrel{\cong}{\longrightarrow}  \Omega'_{f,G},\quad X\mapsto X\vdash\zeta,
\end{equation}
where $\zeta$ is considered as an element in $\Omega'_{f,\id}=\Omega_f$ (recall Definition~\ref{def_Omega'}). 
Fix also a map 
\begin{equation}
\alpha:G_f\longrightarrow \CC^\ast,\quad g\mapsto \alpha_g,
\end{equation}
such that $\alpha_\id=1$ and 
\begin{equation}
\alpha_g\alpha_{g^{-1}}=(-1)^{\frac{1}{2}(N-N_g)(N-N_g+1)},\quad g\in G_f.
\end{equation} 
Such a map $\alpha$ always exists since for each $g$ we may choose $\alpha_g$ as 
\begin{equation}
\alpha_g=\epi\left[\frac{1}{8}(N-N_g)(N-N_g+1)\right].
\end{equation}
For each $g\in G$, let $v_g$ be an element of $\Jac'(f,g)$, such that 
\begin{equation}\label{eq:e_g}
v_g \vdash \zeta = \alpha_g\omega_g,
\end{equation}
where $\omega_g\in\Omega'_{f,g}$ is the residue class of $\widetilde \omega_g\in \Omega^{N_g}(\Fix(g))$ and
\begin{equation}
\widetilde\omega_g:=
\begin{cases}
dx_{i_1}\wedge \dots \wedge dx_{i_{N_g}} & \text{if } I_g=(i_1,\dots, i_{N_g}),\  i_1<\dots <i_{N_g}\\
1_g & \text{if } I_g=\emptyset
\end{cases}.
\end{equation}
Obviously, we have $\omega_{\id}=\zeta$.
\begin{remark}
It might not be necessary to distinguish $\zeta$ and $\omega_\id$, however, we regard $\zeta$ as a ``primitive form" 
(cf. \cite{S1,S2,ST}) 
at the origin of the base space of the ``properly-defined deformation space" of the pair $(f,G)$ 
while we hold $\omega_\id$ as just a $\Jac'(f,\id)$-basis of $\Omega'_{f,\id}$.
\end{remark}

By axiom~\eqref{axiom_vs}, we have $\Jac'(f,\id)=\Jac(f)$. 
Therefore, $v_\id=[1]$ and $v_{\id}\circ v_g=v_g\circ v_{\id}=v_g$ by axioms~\eqref{axiom_algebra} and~\eqref{axiom_idmodule}.
Axiom~\eqref{axiom_idmodule} implies that for all $Y\in\Jac'(f,g)$ there exists $X\in\Jac'(f,\id)=\Jac(f)$ 
represented by a polynomial in $\{x_i\}_{i\in I_g}$ such that $Y=X\circ v_g$.
For any $X\in\Jac'(f,\id)$, we shall often write $X\circ v_g$ as $X|_{\Fix(g)}v_g$ 
where $X|_{\Fix(g)}$ is the image of $X$ under the map $\Jac(f)\longrightarrow \Jac(f^g)$.

\begin{proposition}\label{prop: sum of the fixes}
For a pair $(g,h)$ of elements of $G$ which is not spanning, we have $v_g \circ v_h = 0 \in \Jac'(f,G)$.
\end{proposition}
\begin{proof}
Denote by $[\gamma'_{g,h}(\bx)]$ the element of $\Jac(f^{gh})$ satisfying $v_g \circ v_h = [\gamma'_{g,h}(\bx)] v_{gh}$.
Suppose that $f=f_1\oplus\cdots\oplus f_p$ is a Thom--Sebastiani sum 
such that each $f_\nu$, $\nu=1,\dots, p$ is either of chain type or loop type. 
Without loss of generality, we may assume the coordinate $x_k$, $k\notin I_g \cup I_h \cup I_{gh}$ 
to be a variable of the polynomial $f_1$. Consider the following two cases;
\begin{itemize}
\item[(a)] $f_1=x_1^{a_1}x_2+x_2^{a_2}x_3+\dots+x_{m-1}^{a_{m-1}}x_m+x_m^{a_m}$ is of chain type.
\item[(b)] $f_1=x_1^{a_1}x_2+x_2^{a_2}x_3+\dots+x_{m-1}^{a_{m-1}}x_m+x_m^{a_m}x_1$ is of loop type. 
\end{itemize}

Case (a): 
First, note that $1\notin I_g \cup I_h \cup I_{gh}$.
Consider $(\frac{1}{a_1},0\dots,0)\in \Aut(f_1,G)$ and extend it naturally to the element $\varphi \in \Aut(f,G)$. 
Since $1\notin I_g \cup I_h \cup I_{gh}$, we have $\varphi^*(v_{g'}) = \epi\left[-\frac{1}{a_1}\right] v_{g'}$ for $g'\in\{g,h,gh\}$. 
Axiom~\eqref{axiom_ringauto} yields $\varphi^*([\gamma'_{g,h}(\bx)]) =  \epi\left[-\frac{1}{a_1}\right] [\gamma'_{g,h}(\bx)]$. 
On the other hand, we have $\varphi^*([\gamma'_{g,h}(\bx)]) = [\gamma'_{g,h}(\bx)]$ since $1\notin I_{gh}$.
Hence, $[\gamma'_{g,h}(\bx)]=0$.

Case (b): 
First, note that $1,\dots, m\notin  I_g \cup I_h \cup I_{gh}$. 
Choose an element of $G_{f_1}\backslash G_{f_1}^\SL$, which exists due to Proposition~\ref{prop_GfSL_proper}, 
and extend it naturally to the element $\varphi \in \Aut(f,G)$. 
There exists a complex number $\lambda_\varphi \neq 1$, the determinant of $\varphi$ regarded as an element of ${\rm GL}(N;\CC)$, 
such that $\varphi^*(v_{g'}) = \lambda_\varphi^{-1} v_{g'}$ for $g'\in\{g,h,gh\}$ since $1,\dots, m\notin  I_g \cup I_h \cup I_{gh}$. 
Axiom~\ref{axiom_ringauto} yields $\varphi^*([\gamma'_{g,h}(\bx)]) = \lambda_\varphi^{-1} [\gamma'_{g,h}(\bx)]$. 
On the other hand, we have $\varphi^*([\gamma'_{g,h}(\bx)]) = [\gamma'_{g,h}(\bx)]$ since $1,\dots, m\notin I_{gh}$. 
Hence, $[\gamma'_{g,h}(\bx)]=0$.
\end{proof}

We consider the product $v_g\circ v_h$ for a spanning pair $(g,h)$.
\begin{proposition}\label{proposition_c_{g,h}}
For each spanning pair $(g,h)$ of elements of $G$, there exists $c_{g,h}\in\CC$ such that 
\begin{equation}
v_g\circ v_h = c_{g,h}[H_{g,h}]v_{gh}.
\end{equation}
Moreover, $c_{g,h}$ does not depend on the choice of the subgroup $G$ of $G_f$ containing $g,h$.
\end{proposition}
\begin{proof}
We only need to show the first statement since the second one follows from it together with 
axiom~\eqref{axiom: G-H}, the definition of $v_g$ in \eqref{eq:e_g} and the independence of $H_{g,h}$ from 
a particular choice of $G$.
Based on Lemma~\ref{lemma_cases}, we study which variable in $f_\nu$ can appear in the product structure. 
\begin{lemma}\label{lemma: product explicitly}
Let the notation and the cases be as in Lemma \ref{lemma_cases} above.
There is a polynomial $\gamma_{g,h}(\bx)\in\CC[x_1,\dots, x_N]$ which doesn't depend on $x_{i_1},\dots,x_{i_m}$ such that 
the one of the following holds:
\begin{enumerate}
\item[(i)]
\begin{itemize}
\item[(a)] 
$v_g \circ v_h = [\gamma_{g,h}(\bx)]v_{gh}$.
\item[(b)] 
$v_g \circ v_h = 
\begin{cases}
\left[\gamma_{g,h}(\bx)\cdot\left(x_{i_1}^{a_{i_1}-2} x_{i_2}^{a_{i_2}-1}\cdots x_{i_m}^{a_{i_m}-1}\right)\right] v_{gh} & \text{if }\ l=m\\
\left[\gamma_{g,h}(\bx)\cdot\left( x_{i_1}^{a_{i_1}-2} x_{i_2}^{a_{i_2}-1}\cdots x_{i_l}^{a_{i_l}-1}x_{i_{l+1}}\right) \right] v_{gh} & \text{if }\ l < m
\end{cases},
$
\end{itemize}
\item[(ii)]
\begin{itemize}
\item[(a)]
$v_g \circ v_h =  [\gamma_{g,h}(\bx)]v_{gh}$.
\item[(b)]
$v_g \circ v_h =  [\gamma_{g,h}(\bx)]v_{gh}$.
\item[(c)]
$v_g \circ v_h =  \left[\gamma_{g,h}(\bx)\cdot \left(x_{i_1}^{a_{i_1}-1} x_{i_2}^{a_{i_2}-1}\cdots x_{i_m}^{a_{i_m}-1}\right)\right] v_{gh}$.
\end{itemize}
\end{enumerate}   
Here, we denote by $[\gamma_{g,h}(\bx)]$ the class of $\gamma_{g,h}(\bx)$ in $\Jac(f^{gh})$.
\end{lemma}
\begin{proof}
(i): We may assume $f_\nu=x_1^{a_1}x_2+x_2^{a_2}x_3+\dots+x_m^{a_m}$.
For each $r=1,\dots, m$, there is a unique element $\varphi_r\in \Aut(f_\nu,G)$ 
such that $\varphi_r(x_i)=x_i$ for all $i=r+1,\dots, m$, which is explicitly given by
\begin{align*}
\varphi_r(x_{r}) &:= \epi \left[ \frac{1}{a_r}\right] x_{r},\\
\varphi_r(x_{i}) & := \epi\left[\frac{1}{a_{i}} \left(1 - \frac{1}{a_{i+1}} 
\left( 1 - \dots - \frac{1}{a_{r-1}}\left( 1- \frac{1}{a_r}\right)\right) \right)\right]x_i, \ 1 \le i < r.
\end{align*}
Denote also by $\varphi_p$ its natural extension to $\Aut(f,G)$ and 
by $\lambda_{\varphi_r}\in\CC^\ast$ the determinant of $\varphi_r$ regarded as an element of ${\rm GL}(N;\CC)$.
\begin{itemize}
\item[(a)] 
For each $r=1,\dots, m$, we have $\varphi_r^*(v_g) = v_g$, $\varphi_r^*(v_h) = \lambda_{\varphi_r}^{-1} v_h$ 
and $\varphi_r^*(v_{gh}) = \lambda_{\varphi_r}^{-1} v_{gh}$.
Suppose that a polynomial $\gamma_{g,h}(\bx)\in\CC[x_1,\dots, x_N]$ satisfies $v_g\circ v_h =[\gamma_{g,h}(\bx)]v_{gh}$. 
By axiom~\eqref{axiom_ringauto}, we obtain 
\begin{align*}
[\varphi_r^*(\gamma_{g,h}(\bx))]v_{gh} &=\lambda_{\varphi_r}\varphi_r^*([\gamma_{g,h}(\bx)]v_{gh})
=\lambda_{\varphi_r}\varphi_r^*(v_g\circ v_h)\\
&=\lambda_{\varphi_r}\varphi_r^*(v_g)\circ \varphi_r^*(v_h)=v_g\circ v_h=[\gamma_{g,h}(\bx)]v_{gh},
\end{align*}
and hence $\varphi_r^*([\gamma_{g,h}(\bx)])=[\gamma_{g,h}(\bx)]$ in $\Jac(f^{gh})$. 
In view of the above action of $\varphi_r$ and Proposition~\ref{monomial basis}, 
the polynomial $\gamma_{g,h}(\bx)$ can be chosen so that it does not depend on $x_i$, $i=1,\dots, m$. 
\item[(b)]
For each $r=1,\dots, m$, we have $\varphi_r^*(v_g) = \lambda_{\varphi_r}^{-1} v_g$, $\varphi_r^*(v_h) = \lambda_{\varphi_r}^{-1} v_h$ and $\varphi_r^*(v_{gh}) = v_{gh}$.
Suppose that a polynomial $\gamma'_{g,h}(\bx)\in\CC[x_1,\dots, x_N]$ satisfies $v_g\circ v_h =[\gamma'_{g,h}(\bx)]v_{gh}$. 
By axiom~\eqref{axiom_ringauto}, we obtain 
\begin{align*}
[\varphi_r^*(\gamma'_{g,h}(\bx))]v_{gh} &=\varphi_r^*([\gamma'_{g,h}(\bx)]v_{gh})
=\varphi_r^*(v_g\circ v_h)\\
&=\varphi_r^*(v_g)\circ \varphi_r^*(v_h)=\lambda_{\varphi_r}^{-2}(v_g\circ v_h)=\lambda_{\varphi_r}^{-2}[\gamma'_{g,h}(\bx)]v_{gh},
\end{align*}
and hence $[\varphi_r^*(\gamma'_{g,h}(\bx))]=\lambda_{\varphi_r}^{-2}[\gamma'_{g,h}(\bx)]$ in $\Jac(f^{gh})$.  
In view of the above action of $\varphi_r$ and Proposition~\ref{monomial basis}, 
the polynomial $\gamma'_{g,h}(\bx)$ can be chosen so that
it is divisible by $x_{1}^{a_{1}-2} x_{2}^{a_{2}-1}\cdots x_{m}^{a_m-1}$
if $l=m$ and by $x_{1}^{a_{1}-2} x_{2}^{a_{2}-1}\cdots x_{l}^{a_l-1}x_{l+1}$ if $l<m$. 
\end{itemize}

(ii):
We may assume $f_\nu=x_1^{a_1}x_2+x_2^{a_2}x_3+\dots+x_m^{a_m}x_1$.
For each element $\varphi\in G_{f_\nu}$ regarded as an element of $\Aut(f_\nu,G)$, 
denote also by $\varphi$ its natural extension to $\Aut(f,G)$.
Let $\lambda_\varphi\in\CC^\ast$ be the determinant of $\varphi$ regarded as an element of ${\rm GL}(N;\CC)$.
Note that if $\varphi\ne \id$ then $\varphi(x_i)\ne x_i$ for all $i=1,\dots, m$.
\begin{itemize}
\item[(a)]
For all $\varphi\in G_{f_\nu}$, we have $\varphi^*(v_g) = v_g$, $\varphi^*(v_h) = v_h$ and $\varphi^*(v_{gh}) = v_{gh}$.
Suppose that a polynomial $\gamma_{g,h}(\bx)\in\CC[x_1,\dots, x_N]$ satisfies $v_g\circ v_h =[\gamma_{g,h}(\bx)]v_{gh}$. 
By axiom~\eqref{axiom_ringauto}, we obtain 
\begin{align*}
[\varphi^*(\gamma_{g,h}(\bx))]v_{gh}&=\varphi^*(\gamma_{g,h}(\bx) v_{gh})=\varphi^*(v_g\circ v_h)\\
&= \varphi^*(v_g)\circ \varphi^*(v_h)=v_g\circ v_h =[\gamma_{g,h}(\bx)]v_{gh},
\end{align*}
and hence $[\varphi^*(\gamma_{g,h}(\bx))]=[\gamma_{g,h}(\bx)]$ in $\Jac(f^{gh})$. 
In view of Proposition~\ref{monomial basis}, 
the polynomial $\gamma_{g,h}(\bx)$ can be chosen so that it does not depend on $x_i$, $i=1,\dots, m$. 
\item[(b)]
Suppose that a polynomial $\gamma_{g,h}(\bx)\in\CC[x_1,\dots, x_N]$ satisfies $v_g\circ v_h =[\gamma_{g,h}(\bx)]v_{gh}$. 
Since $1,\dots, m$ do not belong to $I_g\cap I_h$ nor $I_{g,h}$, it is obvious that 
the polynomial $\gamma_{g,h}(\bx)$ can be chosen so that it does not depend on $x_i$, $i=1,\dots, m$. 
\item[(c)]
For all $\varphi\in G_{f_\nu}$, we have $\varphi^*(v_g) = \lambda_\varphi^{-1}v_g$, $\varphi^*(v_h) = \lambda_\varphi^{-1}v_h$ and $\varphi^*(v_{gh}) = v_{gh}$.
Suppose that a polynomial $\gamma'_{g,h}(\bx)\in\CC[x_1,\dots, x_N]$ satisfies $v_g\circ v_h =[\gamma'_{g,h}(\bx)]v_{gh}$. 
By axiom~\eqref{axiom_ringauto}, we obtain 
\begin{align*}
[\varphi^*(\gamma'_{g,h}(\bx))]v_{gh}&=\varphi^*(\gamma'_{g,h}(\bx) v_{gh})=\varphi^*(v_g\circ v_h)\\
&= \varphi^*(v_g)\circ \varphi^*(v_h)=\lambda_\varphi^{-2}(v_g\circ v_h) =\lambda_\varphi^{-2}[\gamma'_{g,h}(\bx)]v_{gh},
\end{align*}
and hence $[\varphi^*(\gamma'_{g,h}(\bx))]=\lambda_\varphi^{-2}[\gamma'_{g,h}(\bx)]$ in $\Jac(f^{gh})$. 
In view of Proposition~\ref{monomial basis}, the polynomial $\gamma'_{g,h}(\bx)$ can be chosen so that 
it is divisible by $x_1^{a_1-1} x_{2}^{a_{2}-1}\cdots x_{m}^{a_m-1}$. 
\end{itemize}
\end{proof}
Now the first statement of the proposition is a direct consequence of Lemma~\ref{lemma: product explicitly}, 
since $H_{g,h}$ is a constant multiple of the product of the monomials in the round brackets there.
We have finished the proof of the proposition.
\end{proof}

By Proposition~\ref{proposition_c_{g,h}}, we may assume that $G=G_f$.
We give some properties of $c_{g,h}$.
\begin{lemma}\label{lemma_gg^-1}
For each $g \in G_f$, we have 
\begin{equation}
c_{g,g^{-1}} = (-1)^{\frac{1}{2}(N-N_g)(N-N_g-1)}\cdot \epi\left[-\frac{1}{2}\age(g)\right].
\end{equation}
\end{lemma}
\begin{proof}
We have
\begin{align*}
\frac{1}{\mu_{f^g}}J_{f,g}([\hess(f^g)]v_g \vdash \zeta,v_{g^{-1}} \vdash \zeta )
&=\frac{\alpha_g\alpha_{g^{-1}}}{\mu_{f^g} }J_{f,g}([\hess(f^g)]\omega_g,\omega_{g^{-1}})\\
&=(-1)^{\frac{1}{2}(N-N_g)(N-N_g-1)}\cdot \epi\left[-\frac{1}{2}\age(g)\right]\cdot |G|.
\end{align*}
On the other hand, by axiom~\eqref{axiom_J_fG} and normalization~\eqref{H and hessians} of $H_{g,h}$, we have
\begin{align*}
\frac{1}{\mu_{f^g}}J_{f,g}([\hess(f^g)]v_g \vdash \zeta,v_{g^{-1}} \vdash \zeta )
&=\frac{1}{\mu_{f^g}}J_{f,id}(\omega_\id ,[\hess(f^g)]v_g \circ v_{g^{-1}} \vdash \zeta )\\
&=\frac{1}{\mu_{f^g}}J_{f,id}(\omega_\id ,c_{g,g^{-1}}[\hess(f^g)H_{g,g^{-1}}]\omega_\id )\\
&=\frac{c_{g,g^{-1}}}{\mu_{f}}J_{f,id}(\omega_\id ,[\hess(f)] \omega_\id)\\
&=c_{g,g^{-1}}|G|.
\end{align*}
\end{proof}
\begin{lemma}\label{lemma_beta=0}
For each pair $(g,h)$ of elements of $G_f$ such that $I_{g,h}= \emptyset$, we have 
\begin{equation}
c_{g,h} c_{h^{-1},g^{-1}} = (-1)^{(N-N_g)(N-N_h)}.
\end{equation}
In particular it follows that $c_{g,h} \neq 0$.
\end{lemma}
\begin{remark}
If $I_{g,h}= \emptyset$ for a pair $(g,h)$ of elements of $G_f$, it is spanning.
\end{remark}
\begin{proof}
We have 
\begin{align*}
&\ v_g \circ  (v_{h} \circ v_{h^{-1}}) \circ v_{g^{-1}} \\
=&\  (-1)^{\frac{1}{2}(N-N_g)(N-N_g-1)+\frac{1}{2}(N-N_h)(N-N_h-1)}\cdot \epi\left[-\frac{1}{2}\age(g)-\frac{1}{2}\age(h)\right][g^*(H_{h,h^{-1}})H_{g,g^{-1}}]v_\id,\\
&\ (v_g \circ  v_{h}) \circ (v_{h^{-1}} \circ v_{g^{-1}}) \\
=& (-1)^{\frac{1}{2}(N-N_{gh})(N-N_{gh}-1)}\epi\left[-\frac{1}{2}\age(gh)\right] c_{g,h} c_{h^{-1},g^{-1}}[H_{gh, (gh)^{-1}}]v_\id.
\end{align*}
The proposition follows from the facts that the product $\circ$ is associative,
$g^*(H_{h,h^{-1}})=H_{h,h^{-1}}$ since $I_{g,h}=\emptyset$, 
$[H_{g,g^{-1}}H_{h,h^{-1}}] = [H_{gh, (gh)^{-1}}]$ in $\Jac(f)$, ${\rm age}(g)+{\rm age}(h)={\rm age}(gh)$ 
since $I_{g,h} = \emptyset$, and $(N-N_g) + (N-N_h) \equiv (N - N_{gh})\ ({\rm mod}\ 2)$ by Proposition~\ref{prop: supergrading agrees}.
\end{proof}
\begin{corollary}\label{cor: factorization}
Let $(g,h)$ be a spanning pair of elements of $G_f$ with the factorization $(g_1,g_2,h_1,h_2)$.
The complex numbers $c_{g_1,h_2}$, $c_{g_2,h_1}$ and $c_{g_1,h_1}$ are non-zero.
\end{corollary}
\begin{proof}
It follows from the fact that $I_{g_1,h_2}=\emptyset$, $I_{g_2,h_1}=\emptyset$ and $I_{g_1,h_1}=\emptyset$. 
\end{proof}
\begin{proposition}\label{prop: 56}
Let $(g,h)$ be a spanning pair of elements of $G_f$ with the factorization $(g_1, g_2,h_1,h_2)$.
We have
\begin{equation}
c_{g,h}=(-1)^{\frac{1}{2}(N-N_{g_2})(N-N_{g_2}-1)}\cdot \epi\left[-\frac{1}{2}\age(g_2)\right]\cdot \frac{c_{g_1,h_1}}{c_{g_1,g_2} c_{h_2,h_1}}.
\end{equation}
In particular, $c_{g,h} \neq 0$.    
\end{proposition}
\begin{proof}
We have
\begin{align*}
v_{g_1} \circ (v_{g_2} \circ v_{h_2}) \circ v_{h_1} &= (-1)^{\frac{1}{2}(N-N_{g_2})(N-N_{g_2}-1)}\cdot \epi\left[-\frac{1}{2}\age(g_2)\right]\cdot  v_{g_1} \circ [H_{g_2,g_2^{-1}}] v_\id \circ v_{h_1} \\
&=  (-1)^{\frac{1}{2}(N-N_{g_2})(N-N_{g_2}-1)}\cdot \epi\left[-\frac{1}{2}\age(g_2)\right] \cdot c_{g_1,h_1}[H_{g_2,g_2^{-1}}]v_{gh}.
\end{align*}
On the other hand, we get:
\begin{align*}
(v_{g_1} \circ v_{g_2}) \circ (v_{h_2} \circ v_{h_1}) &= c_{g_1,g_2} e_{g_1g_2} \circ c_{h_2,h_1} v_{h_1h_2}  \\
&= c_{g_1,g_2} c_{h_2,h_1} c_{g,h} [H_{g,h}] v_{gh}.
\end{align*}
Note that $H_{g,h} = H_{g_2,g_2^{-1}} = H_{h_2,h_2^{-1}}$ by the definition of the factorization $(g_1,g_2,h_1,h_2)$.
By Corollary~\ref{cor: factorization}, we know that $c_{g_1,g_2}$ and $c_{h_2,h_1}$ are non-zero,
which gives the statement.
\end{proof}
Hence, by this proposition, we only have to determine $c_{g,h}$ for all pairs $(g,h)$ of elements of $G_f$ such that $I_{g,h}=\emptyset$. 
Suppose that $f=f_1\oplus\cdots\oplus f_p$ is a Thom--Sebastiani sum 
such that each $f_\nu$, $\nu=1,\dots, p$ is either of chain type or loop type. 
Then, we have a natural isomorphism $G_f\cong G_{f_1}\times \dots \times G_{f_p}$. 
Therefore, it follows that each $g\in G_f$ has the unique expression $g=g_1\cdots g_p$ such that 
$g_i\in G_{f_i}$ for all $i=1,\dots, p$, hence $I_{g_i,g_j}=\emptyset$ if $i\ne j$ and 
$I_g^c=I_{g_1}^c\cup \dots \cup I_{g_p}^c$. 
Under this notation, define $\widetilde v_g$ by
\begin{equation}\label{eq:bar v}
\widetilde v_g:=\widetilde{\varepsilon}_{g_1,\dots, g_p}v_{g_1}\circ\dots \circ v_{g_p}.
\end{equation}
Obviously, $\widetilde v_g$ is a non-zero constant multiple of $v_g$ for all $g\in G_f$.
It is also easy to see that $\widetilde v_g$ does not depend on the choice of ordering in the Thom--Sebastiani sum
and that for a pair $(g,h)$ of elements of $G_f$ with $I_{g,h}=\emptyset$ we have
\begin{equation}
\widetilde v_g\circ \widetilde v_h=\frac{1}{\widetilde \varepsilon_{g,h}}\widetilde v_{gh}.
\end{equation}
\begin{proposition}
For each $g\in G$, we have
\begin{equation}
\widetilde v_g\circ \widetilde v_{g^{-1}}=(-1)^{\frac{1}{2}(N-N_g)(N-N_g-1)}\cdot \epi\left[-\frac{1}{2}\age(g)\right]\cdot 
[H_{g,g^{-1}}]\widetilde v_\id.
\end{equation}
\end{proposition}
\begin{proof}
There is an inductive presentation of $\widetilde v_g$ given by 
\[
\widetilde v_g=
\begin{cases}
v_{g_1} &\quad \text{if}\quad g=g_1,\\
\widetilde{\varepsilon}_{g_1\dots g_i, g_{i+1}}\widetilde v_{g_1\dots g_i} \circ v_{g_{i+1}} & \quad \text{if}\quad g=g_1\dots g_i g_{i+1},\ i=1,\dots, p-1.
\end{cases}
\]
The statement follows by induction from the following calculation:
\begin{align*}
\widetilde v_g\circ \widetilde v_{g^{-1}}
=&\ (\widetilde{\varepsilon}_{g_1\dots g_i, g_{i+1}}\widetilde v_{g_1\dots g_i} \circ v_{g_{i+1}})
\circ (\widetilde{\varepsilon}_{g_1^{-1}\dots g_i^{-1}, g_{i+1}^{-1}}\widetilde v_{g_1^{-1}\dots g_i^{-1}} \circ v_{g_{i+1}^{-1}})\\
= &\ (-1)^{(N-N_{g_1^{-1}\dots g_i^{-1}})(N-N_{g_{i+1}})}\cdot (\widetilde v_{g_1\dots g_i} \circ \widetilde v_{g_1^{-1}\dots g_i^{-1}} )\circ (v_{g_{i+1}}\circ v_{g_{i+1}^{-1}})\\
=&\ (-1)^{(N-N_{g_1\dots g_i})(N-N_{g_{i+1}})+\frac{1}{2}(N-N_{g_1\dots g_i})(N-N_{g_1\dots g_i}-1)+\frac{1}{2}(N-N_{g_{i+1}})(N-N_{g_{i+1}}-1)}\\
&\ \cdot \epi\left[-\frac{1}{2}\age (g_1\dots g_i)-\frac{1}{2}\age(g_{i+1})\right]\cdot 
[H_{g_1\dots g_i,g_1^{-1}\dots g_i^{-1}}H_{g_{i+1},g_{i+1}^{-1}}]\widetilde v_\id\\
=&\ (-1)^{\frac{1}{2}(N-N_g)(N-N_g-1)}\cdot \epi\left[-\frac{1}{2}\age(g)\right]\cdot 
[H_{g,g^{-1}}]\widetilde v_\id.
\end{align*}
\end{proof}
This proposition says that by replacing the map $\alpha:G_f\longrightarrow \CC^\ast$. to the suitable one
we have a new basis $\{\widetilde v_g\}_{g\in G_f}$ instead of $\{v_g\}_{g\in G_f}$. 
To summarize, we finally obtain the following
\begin{corollary}
Let $(g,h)$ be a spanning pair of elements of $G_f$ with the factorization $(g_1,g_2,h_1,h_2)$.
We have 
\begin{equation}
\widetilde v_g\circ \widetilde v_h=(-1)^{\frac{1}{2}(N-N_{g_2})(N-N_{g_2}-1)}\cdot \epi\left[-\frac{1}{2}\age(g_2)\right]\cdot \frac{\widetilde \varepsilon_{g_1,g_2} \widetilde \varepsilon_{h_2,h_1}}{\widetilde \varepsilon_{g_1,h_1}}[H_{g,h}]\widetilde v_{gh}.
\end{equation}
In particular, for any subgroup $G$ of $G_f$, if a $G$-twisted Jacobian algebra of $f$ exists, then it is uniquely determined by the axioms in Definition~\ref{axioms} up to isomorphism.
\end{corollary}
\subsection{Existence}
In this subsection, we prove the existence of a $G$-twisted Jacobian algebra of $f$. 
We first show this when $G=G_f$.
\begin{definition}
Define a $\ZZ/2\ZZ$-graded $\CC$-module $\A'=\A'_{\overline{0}}\oplus \A'_{\overline{1}}$ as follows: 
for each $g\in G_f$, consider a free $\Jac(f^g)$-module $\A'_g$ of rank one generated by a formal letter $\overline v_g$,
\begin{subequations}
\begin{equation}
\A'_g=\Jac(f^g)\overline v_g.
\end{equation}
and set
\begin{equation}
\A'_{\overline{0}}:=\bigoplus_{\substack{g\in G_f\\ N-N_g\equiv 0\ (\text{\rm mod } 2)}}\A'_g,
\quad
\A'_{\overline{1}}:=\bigoplus_{\substack{g\in G_f\\ N-N_g\equiv 1\ (\text{\rm mod } 2)}}\A'_g.
\end{equation}
\end{subequations}
\end{definition}
By definition, axiom~\eqref{axiom_vs} trivially holds for $\A'$.
\begin{definition}
For a spanning pair $(g,h)$ of elements of $G_f$ with the factorization $(g_1,g_2,h_1,h_2)$, set
\begin{equation}\label{bar c}
\overline c_{g,h} := 
(-1)^{\frac{1}{2}(N-N_{g_2})(N-N_{g_2}-1)}\cdot \epi \left[-\frac{1}{2}\age(g_2)\right]\cdot 
\frac{\widetilde{\varepsilon}_{g_1,g_2}\widetilde{\varepsilon}_{h_2,h_1}}{\widetilde{\varepsilon}_{g_1,h_1}}.
\end{equation}
\end{definition}
It is also easy to see that
\begin{subequations}
\begin{eqnarray}
&\overline c_{g,\id}=1=\overline c_{\id,g},\ &g\in G_f,\\
&\overline c_{g,g^{-1}} = (-1)^{\frac{1}{2}(N-N_g)(N-N_g - 1)}\cdot \epi \left[-\frac{1}{2}\age(g)\right],\quad &g\in G_f,\\
&\overline c_{g,h}=\widetilde{\varepsilon}_{g,h}^{-1},\quad &g,h\in G_f,\ I_{g,h}= \emptyset.
\end{eqnarray}
\end{subequations}
\begin{definition}
For each $g,h\in G_f$, define an element of $\A'_{gh}$ by 
\begin{equation}
\overline v_g \circ \overline v_h:=
\begin{cases}
\overline c_{g,h} \left[H_{g,h} \right] \overline v_{gh} & \text{if the pair $(g,h)$ is spanning}\\
0 & \text{otherwise}
\end{cases}.
\end{equation}
It is clear that $\overline v_\id\circ\overline v_g=\overline v_g=\overline v_g\circ\overline v_\id$ 
since $I_{\id,g}=I_{g,\id}=\emptyset$ and hence $[H_{\id,g}]=[H_{g,\id}]=1$.
\end{definition}
\begin{proposition}\label{prop:Gtwistedcomm}
For a spanning pair $(g,h)$ of elements of $G_f$ with the factorization $(g_1,g_2,h_1,h_2)$, we have 
\begin{equation}
\overline c_{g,h} = (-1)^{(N-N_g)(N-N_h)}\cdot \epi\left[-\age(g_2)\right]\cdot \overline c_{h,g}.
\end{equation}
Hence, we have 
\begin{equation}
\overline v_g\circ \overline v_h=(-1)^{(N-N_g)(N-N_h)}
\cdot\left(\epi\left[-\age(g_2)\right] \overline v_h\circ \overline v_g\right).
\end{equation}
\end{proposition}
\begin{proof}
We have
\begin{align*}
\overline c_{g,h} =& 
(-1)^{\frac{1}{2}(N-N_{g_2})(N-N_{g_2}-1)}\cdot  \epi \left[-\frac{1}{2}\age(g_2)\right]\cdot \frac{\widetilde{\varepsilon}_{g_1,g_2}\widetilde{\varepsilon}_{h_2,h_1}}{\widetilde{\varepsilon}_{g_1,h_1}} \\
=&(-1)^{(N-N_{g_1})(N-N_{g_2})+(N-N_{h_1})(N-N_{h_2})-(N-N_{g_1})(N-N_{h_1})+(N-N_{g_2})}\cdot\epi \left[-\age(g_2)\right]\\
&\cdot (-1)^{\frac{1}{2}(N-N_{h_2})(N-N_{h_2}-1)} \cdot \epi \left[-\frac{1}{2}\age(h_2)\right]\cdot 
\frac{\widetilde{\varepsilon}_{h_1,h_2}\widetilde{\varepsilon}_{g_2,g_1}}{\widetilde{\varepsilon}_{h_1,g_1}}\\
=& (-1)^{(N-N_g)(N-N_h)}\cdot \epi \left[-\age(g_2)\right]\cdot \overline c_{h,g},
\end{align*}
where we used that $h_2=g_2^{-1}$, $N-N_{g_2}=\age(g_2)+\age(h_2)$ and Proposition~\ref{prop: supergrading agrees}.
\end{proof}
\begin{proposition}\label{prop:ass bar v} 
For each $g,g',g''\in G_f$, we have 
\begin{equation}\label{eq:associativity}
(\overline v_g\circ \overline v_{g'})\circ \overline v_{g''}=\overline v_g\circ (\overline v_{g'}\circ \overline v_{g''}).
\end{equation}
\end{proposition}
\begin{proof}
First, we show the following
\begin{lemma}\label{lem:ggg}
For $g,g',g''\in G_f$, suppose that $(g,g')$ and $(gg',g'')$ are spanning pairs with $I_{g,g'}\subseteq I_{g''}$.
\begin{enumerate}
\item
There exist $g_1, g_2, g_3, g_1', g_2', g_3', g_1'', g_2'', g_3''\in G_f$ such that 
\begin{equation}
g=g_1g_2g_3,\ g'=g_1'g_2'g_3',\ g''=g_1''g_2''g_3'',\quad g_1'g_1''=\id, g_2g_2''=\id,\ g_3g_3'=\id,
\end{equation}
and $(g_1g_2,g_3,g_1'g_2',g_3')$ is the factorization of $(g,g')$ and 
$(g_1g_2',g_2g_1',g_3'',g_1''g_2'')$ is the facotrization of $(gg',g'')$.
\item
The pairs $(g',g'')$ and $(g,g'g'')$ are spanning such that $I_{g',g''}\subseteq I_g$.
\end{enumerate}
\end{lemma}
\begin{proof}
(i) Similarly to the presentation of \eqref{eq: gh explicit form},
the elements $g,g',g''$ satisfying the conditions can be expressed, in the multiplicative form, as follows: 
\begin{equation}
\begin{matrix}
g & = & g_1 & \cdot & g_2 & \cdot & \id & \cdot & \id & \cdot & g_3 & \cdot & \id \\
g' & = & \id & \cdot & \id & \cdot & g_1' & \cdot & g_2'& \cdot & g_3' & \cdot & \id \\
g''& = & \id & \cdot & g_2'' & \cdot & g_1'' & \cdot & \id & \cdot & \id & \cdot & g_3''
\end{matrix}.
\end{equation}
(ii) By the above presentation, it is easy to see that $(g,g')$ and $(gg',g'')$ are spanning pairs.
It follows from $g_1'g_1''=\id$ that $I_{g',g''}\subseteq I_g$. 
\end{proof}

\begin{lemma}
The LHS of \eqref{eq:associativity} is non-zero if and only if the RHS of \eqref{eq:associativity} is non-zero.
\end{lemma}
\begin{proof}
By Proposition~\ref{proposition: H_gh is a jac element} (iii), the LHS of \eqref{eq:associativity} is non-zero only 
if both pairs $(g,g')$ and $(gg',g'')$ are spanning and $I_{g,g'} \subseteq I_{g''}$ and the RHS of \eqref{eq:associativity} is non-zero only if both pairs $(g,g'g'')$ and $(g',g'')$ are spanning and $I_{g',g''} \subseteq I_{g}$.
Lemma~\ref{lem:ggg} together with Proposition~\ref{prop:Gtwistedcomm} yields the statement.
\end{proof}
\begin{lemma}
Let the notations be as above. 
We have 
\begin{equation}
H_{g,g'}=H_{g_3,g_3'},\ H_{gg',g''}=H_{g_2g_1',g_2''g_1''}, \ H_{g,g'g''}=H_{g_2g_3,g_2''g_3'},\ H_{g',g''}=H_{g_1',g_1''},
\end{equation}
and hence $\left[H_{g,g'} H_{gg',g''}\right] = \left[H_{g,g'g''}H_{g',g''}\right]$ in $\Jac(f^{gg'g''})$.
\end{lemma}
\begin{proof}
The first statement follows from the definition of $H_{g,h}$ and the second one does from Proposition~\ref{proposition: H_gh is a jac element} (ii).
\end{proof}
Therefore, we only have to show the following
\begin{lemma}
Let the notations be as above. we have
\begin{equation}
\overline c_{g,g'}\overline c_{gg',g''}=\overline c_{g,g'g''}\overline c_{g',g''}.
\end{equation}
\end{lemma}
\begin{proof}
It follows from the definition~\eqref{bar c} that 
\begin{align*}
&\overline c_{g,g'}=
(-1)^{\frac{1}{2}(N-N_{g_3})(N-N_{g_3}-1)}\cdot \epi \left[-\frac{1}{2}\age(g_3)\right]\cdot
\frac{\widetilde{\varepsilon}_{g_1g_2,g_3}\widetilde{\varepsilon}_{g_3',g_1'g_2'}}{\widetilde{\varepsilon}_{g_1g_2,g_1'g_2'}}, \\
&\overline c_{gg',g''}=
(-1)^{\frac{1}{2}(N-N_{g_2g_1'})(N-N_{g_2g_1'}-1)}\cdot \epi \left[-\frac{1}{2}\age(g_2g_1')\right]\cdot
\frac{\widetilde{\varepsilon}_{g_1g_2',g_2g_1'}\widetilde{\varepsilon}_{g_2''g_1'',g_3''}}{\widetilde{\varepsilon}_{g_1g_2',g_3''}}, \\
&\overline c_{g,g'g''}=
(-1)^{\frac{1}{2}(N-N_{g_2g_3})(N-N_{g_2g_3}-1)}\cdot \epi \left[-\frac{1}{2}\age(g_2g_3)\right]\cdot
\frac{\widetilde{\varepsilon}_{g_1,g_2g_3}\widetilde{\varepsilon}_{g_2'g_3',g_2''g_3''}}{\widetilde{\varepsilon}_{g_1,g_2'g_3''}},\\
&\overline c_{g',g''}=
(-1)^{\frac{1}{2}(N-N_{g_1'})(N-N_{g_1'}-1)}\cdot \epi \left[-\frac{1}{2}\age(g_1')\right]\cdot
\frac{\widetilde{\varepsilon}_{g_2'g_3',g_1'}\widetilde{\varepsilon}_{g_1'',g_2''g_3''}}{\widetilde{\varepsilon}_{g_2'g_3',g_2''g_3''}}.
\end{align*}
Since all $I_{g_i}^c$, $I_{g_i'}^c$ and $I_{g_i''}^c$ are mutually disjoint, we get 
\begin{align*}
\overline c_{g,g'}\overline c_{gg',g''}=&
(-1)^{\frac{1}{2}(N-N_{g_3})(N-N_{g_3}-1)+\frac{1}{2}(N-N_{g_2g_1'})(N-N_{g_2g_1'}-1)}\\
&\cdot \epi \left[-\frac{1}{2}\age(g_3)-\frac{1}{2}\age(g_2g_1')\right]\cdot
\frac{\widetilde{\varepsilon}_{g_1g_2,g_3}\widetilde{\varepsilon}_{g_3',g_1'g_2'}}{\widetilde{\varepsilon}_{g_1g_2,g_1'g_2'}}\frac{\widetilde{\varepsilon}_{g_1g_2',g_2g_1'}\widetilde{\varepsilon}_{g_2''g_1'',g_3''}}{\widetilde{\varepsilon}_{g_1g_2',g_3''}} \\
=&(-1)^{\frac{1}{2}(N-N_{g_3})(N-N_{g_3}-1)+\frac{1}{2}(N-N_{g_2}+N-N_{g_1'})(N-N_{g_2}+N-N_{g_1'}-1)}\\
&\cdot \epi \left[-\frac{1}{2}\age(g_3)-\frac{1}{2}\age(g_2)-\frac{1}{2}\age(g_1')\right]\\
&\cdot 
\frac{\widetilde{\varepsilon}_{g_1,g_2}\widetilde{\varepsilon}_{g_1,g_2,g_3}\widetilde{\varepsilon}_{g_3',g_1',g_2'}\widetilde{\varepsilon}_{g_1',g_2'}\widetilde{\varepsilon}_{g_1,g_2'}\widetilde{\varepsilon}_{g_1,g_2',g_1',g_2}\widetilde{\varepsilon}_{g_1',g_2}\widetilde{\varepsilon}_{g_1'',g_2''}\widetilde{\varepsilon}_{g_1'',g_2'',g_3''}}{\widetilde{\varepsilon}_{g_1,g_2}\widetilde{\varepsilon}_{g_1,g_2,g_1',g_2'}\widetilde{\varepsilon}_{g_1',g_2'}
\widetilde{\varepsilon}_{g_1,g_2'}\widetilde{\varepsilon}_{g_1,g_2',g_3''}}\\ 
=&(-1)^{\frac{1}{2}\left((N-N_{g_3})^2-(N-N_{g_3})+(N-N_{g_2})^2-(N-N_{g_2})+(N-N_{g_1'})^2-(N-N_{g_1'})+2(N-N_{g_2})(N-N_{g_1'})\right)}\\
&\cdot \epi \left[-\frac{1}{2}\age(g_3)-\frac{1}{2}\age(g_2)-\frac{1}{2}\age(g_1')\right]\\
&\cdot
\frac{\widetilde{\varepsilon}_{g_1,g_2,g_3}\widetilde{\varepsilon}_{g_3',g_1',g_2'}\widetilde{\varepsilon}_{g_1,g_2',g_1',g_2}\widetilde{\varepsilon}_{g_1',g_2}\widetilde{\varepsilon}_{g_1'',g_2''}\widetilde{\varepsilon}_{g_1'',g_2'',g_3''}}{\widetilde{\varepsilon}_{g_1,g_2,g_1',g_2'}\widetilde{\varepsilon}_{g_1,g_2',g_3''}},
\end{align*}
and
\begin{align*}
\overline c_{g,g'g''}\overline c_{g',g''}=&
(-1)^{\frac{1}{2}(N-N_{g_2g_3})(N-N_{g_2g_3}-1)+\frac{1}{2}(N-N_{g_1'})(N-N_{g_1'}-1)}\\
&\cdot \epi \left[-\frac{1}{2}\age(g_2g_3)-\frac{1}{2}\age(g_1')\right]\cdot
\frac{\widetilde{\varepsilon}_{g_1,g_2g_3}\widetilde{\varepsilon}_{g_2''g_3',g_2'g_3''}}{\widetilde{\varepsilon}_{g_1,g_2'g_3''}}\frac{\widetilde{\varepsilon}_{g_2'g_3',g_1'}\widetilde{\varepsilon}_{g_1'',g_2''g_3''}}{\widetilde{\varepsilon}_{g_2'g_3',g_2''g_3''}}\\
=&(-1)^{\frac{1}{2}(N-N_{g_2}+N-N_{g_3})(N-N_{g_2}+N-N_{g_3}-1)+\frac{1}{2}(N-N_{g_1'})(N-N_{g_1'}-1)}\\
&\cdot \epi \left[-\frac{1}{2}\age(g_3)-\frac{1}{2}\age(g_2)-\frac{1}{2}\age(g_1')\right]\\
&\cdot
\frac{\widetilde{\varepsilon}_{g_1,g_2,g_3}\widetilde{\varepsilon}_{g_2,g_3}\widetilde{\varepsilon}_{g_2'',g_3'}\widetilde{\varepsilon}_{g_2'',g_3',g_2',g_3''}\widetilde{\varepsilon}_{g_2',g_3''}\widetilde{\varepsilon}_{g_2',g_3'}\widetilde{\varepsilon}_{g_2',g_3',g_1'}\widetilde{\varepsilon}_{g_1'',g_2'',g_3''}\widetilde{\varepsilon}_{g_2'',g_3''}}{\widetilde{\varepsilon}_{g_1,g_2',g_3''}\widetilde{\varepsilon}_{g_2',g_3''}\widetilde{\varepsilon}_{g_2',g_3'}\widetilde{\varepsilon}_{g_2',g_3',g_2'',g_3''}\widetilde{\varepsilon}_{g_2'',g_3''}}\\
=&(-1)^{\frac{1}{2}\left((N-N_{g_3})^2-(N-N_{g_3})+(N-N_{g_2})^2-(N-N_{g_2})+(N-N_{g_1'})^2-(N-N_{g_1'})+2(N-N_{g_2})(N-N_{g_3})\right)}\\
&\cdot \epi \left[-\frac{1}{2}\age(g_3)-\frac{1}{2}\age(g_2)-\frac{1}{2}\age(g_1')\right]\\
&\cdot
\frac{\widetilde{\varepsilon}_{g_1,g_2,g_3}\widetilde{\varepsilon}_{g_2,g_3}\widetilde{\varepsilon}_{g_2'',g_3'}\widetilde{\varepsilon}_{g_2'',g_3',g_2',g_3''}\widetilde{\varepsilon}_{g_2',g_3',g_1'}\widetilde{\varepsilon}_{g_1'',g_2'',g_3''}}{\widetilde{\varepsilon}_{g_1,g_2',g_3''}\widetilde{\varepsilon}_{g_2',g_3',g_2'',g_3''}}.\\
\end{align*}
Therefore, we only have to show that
\begin{align*}
&\ (-1)^{(N-N_{g_2})(N-N_{g_1'})}\cdot
\frac{\widetilde{\varepsilon}_{g_1,g_2,g_3}\widetilde{\varepsilon}_{g_3',g_1',g_2'}\widetilde{\varepsilon}_{g_1,g_2',g_1',g_2}\widetilde{\varepsilon}_{g_1',g_2}\widetilde{\varepsilon}_{g_1'',g_2''}\widetilde{\varepsilon}_{g_1'',g_2'',g_3''}}{\widetilde{\varepsilon}_{g_1,g_2,g_1',g_2'}\widetilde{\varepsilon}_{g_1,g_2',g_3''}} \\
=&\ 
(-1)^{(N-N_{g_2})(N-N_{g_3})}\cdot
\frac{\widetilde{\varepsilon}_{g_1,g_2,g_3}\widetilde{\varepsilon}_{g_2,g_3}\widetilde{\varepsilon}_{g_2'',g_3'}\widetilde{\varepsilon}_{g_2'',g_3',g_2',g_3''}\widetilde{\varepsilon}_{g_2',g_3',g_1'}\widetilde{\varepsilon}_{g_1'',g_2'',g_3''}}{\widetilde{\varepsilon}_{g_1,g_2',g_3''}\widetilde{\varepsilon}_{g_2',g_3',g_2'',g_3''}}.
\end{align*}
Since $g_1'g_1''=\id$, $g_2g_2''=\id$ and $g_3g_3'=\id$, we have $I_{g_1'}^c=I_{g_1''}^c, I_{g_2}^c=I_{g_2''}^c$ and $I_{g_3}^c=I_{g_3'}^c$. 
We also have that $\widetilde{\varepsilon}_{\bullet}^2=1$ for any expression $\bullet$. 
Hence, the problem is reduced to show the following equation:
\begin{align*}
&\ (-1)^{(N-N_{g_2})(N-N_{g_1'})}\cdot
\frac{\widetilde{\varepsilon}_{g_1,g_2,g_3}\widetilde{\varepsilon}_{g_3,g_1',g_2'}
\widetilde{\varepsilon}_{g_1,g_2',g_1',g_2}\widetilde{\varepsilon}_{g_1',g_2,g_3''}}{\widetilde{\varepsilon}_{g_1,g_2,g_1',g_2'}\widetilde{\varepsilon}_{g_1,g_2',g_3''}} \\
=&\ 
(-1)^{(N-N_{g_2})(N-N_{g_3})}\cdot
\frac{\widetilde{\varepsilon}_{g_1,g_2,g_3}\widetilde{\varepsilon}_{g_2,g_3,g_2',g_3''}\widetilde{\varepsilon}_{g_2',g_3,g_1'}\widetilde{\varepsilon}_{g_1',g_2,g_3''}}{\widetilde{\varepsilon}_{g_1,g_2',g_3''}\widetilde{\varepsilon}_{g_2',g_3,g_2,g_3''}}.
\end{align*}
Recall also  that $\widetilde \varepsilon_\bullet$ is the signature of a permutation $\sigma_\bullet$ based on the expression $\bullet$ 
(see Definition~\ref{def: epsilon}), and hence we get a suitable sign by interchanging two indexes, for example, 
$\widetilde{\varepsilon}_{g_3,g_1',g_2'}=(-1)^{(N-N_{g_1'})(N-N_{g_2'})}\widetilde{\varepsilon}_{g_3,g_2',g_1'}$.
The LHS of the above equation is given by 
\begin{align*}
&\ (-1)^{(N-N_{g_2})(N-N_{g_1'})}\cdot
\frac{\widetilde{\varepsilon}_{g_1,g_2,g_3}\widetilde{\varepsilon}_{g_3,g_1',g_2'}
\widetilde{\varepsilon}_{g_1,g_2',g_1',g_2}\widetilde{\varepsilon}_{g_1',g_2,g_3''}}{\widetilde{\varepsilon}_{g_1,g_2,g_1',g_2'}\widetilde{\varepsilon}_{g_1,g_2',g_3''}} \\
=&\ (-1)^{(N-N_{g_2})(N-N_{g_1'})}\cdot\frac{\widetilde{\varepsilon}_{g_1,g_2,g_3}(-1)^{(N-N_{g_1'})(N-N_{g_2'})}\widetilde{\varepsilon}_{g_3,g_2',g_1'}\widetilde{\varepsilon}_{g_1',g_2,g_3''}}{\widetilde{\varepsilon}_{g_1,g_2',g_3''}}\\
&\quad \cdot\frac{(-1)^{(N-N_{g_2})(N-N_{g_1'})+(N-N_{g_2})(N-N_{g_2'})+(N-N_{g_2'})(N-N_{g_1'})}\widetilde{\varepsilon}_{g_1,g_2,g_1',g_2'}}{\widetilde{\varepsilon}_{g_1,g_2,g_1',g_2'}} \\
=&\ (-1)^{(N-N_{g_2})(N-N_{g_2'})}\cdot \frac{\widetilde{\varepsilon}_{g_1,g_2,g_3}\widetilde{\varepsilon}_{g_3,g_2',g_1'}
\widetilde{\varepsilon}_{g_1',g_2,g_3''}}{\widetilde{\varepsilon}_{g_1,g_2',g_3''}},
\end{align*}
while the RHS is given by 
\begin{align*}
&(-1)^{(N-N_{g_2})(N-N_{g_3})}\cdot
\frac{\widetilde{\varepsilon}_{g_1,g_2,g_3}\widetilde{\varepsilon}_{g_2,g_3,g_2',g_3''}\widetilde{\varepsilon}_{g_2',g_3,g_1'}\widetilde{\varepsilon}_{g_1',g_2,g_3''}}{\widetilde{\varepsilon}_{g_1,g_2',g_3''}\widetilde{\varepsilon}_{g_2',g_3,g_2,g_3''}}\\
=&\ (-1)^{(N-N_{g_2})(N-N_{g_3})}\cdot\frac{\widetilde{\varepsilon}_{g_1,g_2,g_3}(-1)^{(N-N_{g_2'})(N-N_{g_3})}\widetilde{\varepsilon}_{g_3,g_2',g_1'}\widetilde{\varepsilon}_{g_1',g_2,g_3''}}{\widetilde{\varepsilon}_{g_1,g_2',g_3''}}\\
&\quad \cdot\frac{(-1)^{(N-N_{g_2})(N-N_{g_3})+(N-N_{g_2})(N-N_{g_2'})+(N-N_{g_2'})(N-N_{g_3})}\widetilde{\varepsilon}_{g_2',g_3,g_2,g_3''}}{\widetilde{\varepsilon}_{g_2',g_3,g_2,g_3''}}\\
=&\ (-1)^{(N-N_{g_2})(N-N_{g_2'})}\cdot
 \frac{\widetilde{\varepsilon}_{g_1,g_2,g_3}\widetilde{\varepsilon}_{g_3,g_2',g_1'}\widetilde{\varepsilon}_{g_1',g_2,g_3''}}{\widetilde{\varepsilon}_{g_1,g_2',g_3''}},
\end{align*}
which coincides with the LHS.
\end{proof}
We have finished the proof of the proposition.
\end{proof}
Now, it is possible to equip $\A'$ with a structure of $\ZZ/2\ZZ$-graded $\CC$-algebra.
\begin{definition}
Define a $\CC$-bilinear map $\circ:\A'\otimes_\CC \A'\longrightarrow \A'$ by
setting, for each $g,h\in G_f$ and $\phi(\bx)$, $\psi(\bx)\in\CC[x_1,\dots, x_N]$,
\begin{equation}
\left([\phi(\bx)]\overline v_g\right) \circ\left([\psi(\bx)]\overline v_h\right)
:=\overline c_{g,h} \left[\phi(\bx)\psi(\bx)H_{g,h} \right] \overline v_{gh}.
\end{equation}
\end{definition}
It is easy to see that the map $\circ$ is well-defined by Proposition~\ref{proposition: H_gh is a jac element} (iii).
\begin{proposition}
The map $\circ$ equips $\A'$ with a structure of $\ZZ/2\ZZ$-graded $\CC$-algebra with the identity $\overline v_\id$, 
which satisfies axiom~\eqref{axiom_algebra}.
\end{proposition}
\begin{proof}
The associativity of the product follows from Proposition~\ref{prop:ass bar v}.
It is obvious by Proposition~\ref{prop: supergrading agrees} 
that $\A_{\overline i}'\circ\A_{\overline j}'\subset \A_{\overline {i+j}}'$ for all $\overline i,\overline j\in\ZZ/2\ZZ$.
It is also clear by the definition of the map $\circ$ above 
that the natural surjective maps $\Jac(f)\longrightarrow \Jac(f^g)$, $g\in G_f$ equip $\A'$ 
with a structure of $\Jac(f)$-module, which coincides with the product map $\circ:\A_\id'\otimes_\CC \A_g'\longrightarrow \A_g'$.
\end{proof}
Take a nowhere vanishing $N$-form $dx_1\wedge \dots\wedge dx_N$
and set $\zeta:=[dx_1\wedge \dots\wedge dx_N]\in\Omega_f$.
For each $g\in G_f$, let $\omega_g\in\Omega'_{f,g}$ be the residue class of $\widetilde \omega_g\in \Omega^{N_g}(\Fix(g))$ 
where
\begin{equation}\label{omega tilde}
\widetilde\omega_g:=
\begin{cases}
dx_{i_1}\wedge \dots \wedge dx_{i_{N_g}} & \text{if } I_g=(i_1,\dots, i_{N_g}),\  i_1<\dots <i_{N_g}\\
1_g & \text{if } I_g=\emptyset
\end{cases}.
\end{equation}
Obviously, we have $\omega_{\id}=\zeta$.
\begin{definition}
Define a $\CC$-bilinear map $\vdash:\A'\otimes_\CC \Omega'_{f,G_f}\longrightarrow \Omega'_{f,G_f}$ by
setting, for each $g,h\in G_f$ and $\phi(\bx)$, $\psi(\bx)\in\CC[x_1,\dots, x_N]$,
\begin{equation}
\left([\phi(\bx)]\overline v_g\right) \vdash \left([\psi(\bx)]\omega_h\right)
:=\frac{\overline\alpha_{gh}\overline c_{g,h}}{\overline\alpha_{h}} \left[\phi(\bx)\psi(\bx)H_{g,h} \right] \omega_{gh},
\end{equation}
where $\overline \alpha:G\longrightarrow \CC^\ast$, $g\mapsto \overline\alpha_g$ is a map given by 
\begin{equation}\label{alphag}
\overline \alpha_g:=\epi\left[\frac{1}{8}(N-N_g)(N-N_g+1)\right].
\end{equation}
\end{definition}
\begin{remark}
The map $\overline \alpha:G\longrightarrow \CC^\ast$ satisfies 
$\overline\alpha_\id=1$ and 
\begin{equation}\label{eq:aa}
\overline \alpha_g\overline\alpha_{g^{-1}}=(-1)^{\frac{1}{2}(N-N_g)(N-N_g+1)},\quad g\in G_f.
\end{equation} 
\end{remark}
The map $\vdash$ induces an isomorphism $\vdash\zeta:\A'\longrightarrow  \Omega'_{f,G_f}$ of $\ZZ/2\ZZ$-graded $\CC$-modules:
\begin{equation}
\vdash\zeta:\A_g'\longrightarrow  \Omega'_{f,g},\quad [\phi(\bx)]\overline v_g\mapsto [\phi(\bx)]\overline v_g\vdash\zeta=\overline\alpha_g[\phi(\bx)]\omega_g,
\end{equation}
Note that for each $g,h\in G_f$ and $\phi(\bx)$, $\psi(\bx)\in\CC[x_1,\dots, x_N]$ we have 
\begin{equation}
\left([\phi(\bx)]\overline v_g\right) \vdash \left([\psi(\bx)]\omega_h\right)
=\left(\left([\phi(\bx)]\overline v_g\right) \circ\left([\psi(\bx)]\overline v_h\right)\right)\vdash\zeta,
\end{equation}
by which we obtain the following
\begin{proposition}
The map $\vdash:\A'\otimes_\CC \Omega'_{f,G_f}\longrightarrow \Omega'_{f,G_f}$ satisfies axiom~\eqref{axiom_=JacOmega}
in Definition~\ref{axioms}.
\end{proposition}
On $\A'$ we have the action of $\varphi\in \Aut(f,G)$ induced by the isomorphism $\vdash\zeta:\A'\longrightarrow  \Omega'_{f,G_f}$, which is denoted by $\varphi^*$. 
We also use the notation of \eqref{conj-action}.
\begin{proposition}
Axiom~{\rm (iv)} is satisfied by $\A'$, namely,
axioms~\eqref{axiom_ringauto} and \eqref{axiom: G-twisted comm} hold.
\end{proposition}
\begin{proof}
Let $(g,h)$ be a spanning pair of elements of $G_f$ with the factorization $(g_1,g_2,h_1,h_2)$ 
and $\varphi$ an element of $\Aut(f,G)$.
For simplicity, set $g':=\varphi\circ g\circ \varphi^{-1}$, $h':=\varphi\circ h\circ \varphi^{-1}$, 
$g_i':=\varphi\circ g_i\circ \varphi^{-1}$ 
and $h_i':=\varphi\circ h_i\circ \varphi^{-1}$ for $i=1,2$.
Note that the pair $(g',h')$ is a spanning pair with the factorization $(g_1',g_2',h_1',h_2')$
since $\varphi$ is a $\CC$-algebra automorphism of $\CC[x_1,\dots, x_N]$, 
which induces a bi-regular map $\varphi:\left(\Fix(g_i')\right)\longrightarrow \Fix(g_i)$. 
It also follows that there exist $\lambda_\varphi, \lambda_{\varphi_{g_i}}, \lambda_{\varphi_{h_i}}\in\CC^\ast$, $i=1,2$ such that  
\[
\varphi^\ast(\widetilde \omega_\id)=\lambda_{\varphi}\widetilde \omega_\id,\quad 
\varphi^\ast(\widetilde \omega_{g_i})=\lambda_{\varphi_{g_i}}\widetilde \omega_{g_i'},\ 
\varphi^\ast(\widetilde \omega_{h_i})=\lambda_{\varphi_{h_i}}\widetilde \omega_{h_i'},\ i=1,2,
\]
(see \eqref{omega tilde} for the definition of $\widetilde \omega_{g}$)
and that, by~\eqref{alphag}, 
$\overline\alpha_{g'}=\overline\alpha_g$, $\overline\alpha_{h'}=\overline\alpha_h$, $\overline\alpha_{g_i'}=\overline\alpha_{g_i}$ and $\overline\alpha_{h_i'}=\overline\alpha_{h_i}$ for $i=1,2$. 
For each $\phi(\bx)\in\CC[x_1,\dots, x_N]$, we have 
\[
\varphi^\ast([\phi(\bx)]\overline v_g)=[\varphi^\ast\phi(\bx)]\varphi^\ast(\overline v_g),
\]
since 
\begin{align*}
&\varphi^\ast([\phi(\bx)]\overline v_g)\vdash\varphi^\ast(\zeta)
=\varphi^\ast([\phi(\bx)]\overline v_g\vdash \zeta)
=\varphi^\ast(\overline \alpha_g[\phi(\bx)]\omega_g)\\
&=\overline \alpha_g[\varphi^\ast\phi(\bx)]\varphi^\ast(\omega_g)
=\frac{\overline \alpha_g}{\overline\alpha_{g'}}
\left([\varphi^\ast\phi(\bx)]\varphi^\ast(\overline v_g)\right)\vdash\varphi^\ast(\zeta)
=\left([\varphi^\ast\phi(\bx)]\varphi^\ast(\overline v_g)\right)\vdash\varphi^\ast(\zeta).
\end{align*}
Therefore, we only need to show that 
$\varphi^\ast(\overline v_g)\circ \varphi^\ast(\overline v_h)=\varphi^\ast(\overline v_g\circ \overline v_h)$.
It easily follows that 
\[
\varphi^\ast(\overline v_\id)=\overline v_\id,\quad 
\varphi^\ast(\overline v_{g_i})=\frac{\lambda_{\varphi_{g_i}}}{\lambda_{\varphi}}\overline v_{g_i'}, \quad 
\varphi^\ast(\overline v_{h_i})=\frac{\lambda_{\varphi_{h_i}}}{\lambda_{\varphi}}\overline v_{h_i'}, 
\ i=1,2,
\] 
since 
$\varphi^\ast(\overline v_\id)\vdash \varphi^\ast(\zeta)=\varphi^\ast(\overline v_\id\vdash\zeta)=\varphi^\ast(\zeta)$ 
and 
\begin{align*}
&(\lambda_{\varphi_{g_i}}\overline v_{g_i'})\vdash \zeta
=\lambda_{\varphi_{g_i}}\overline\alpha_{g_i'}\omega_{g_i'}
=\varphi^\ast(\overline\alpha_{g_i}\omega_{g_i})
=\varphi^\ast(\overline v_{g_i})\vdash\varphi^\ast(\zeta)
=\left(\lambda_{\varphi}\varphi^\ast(\overline v_{g_i})\right)\vdash \zeta,\\
&(\lambda_{\varphi_{h_i}}\overline v_{h_i'})\vdash \zeta
=\lambda_{\varphi_{h_i}}\overline\alpha_{h_i'}\omega_{h_i'}
=\varphi^\ast(\overline\alpha_{h_i}\omega_{h_i})
=\varphi^\ast(\overline v_{h_i})\vdash\varphi^\ast(\zeta)
=\left(\lambda_{\varphi}\varphi^\ast(\overline v_{h_i})\right)\vdash \zeta.
\end{align*}
\begin{lemma}
We have
\begin{equation}
\varphi^*(\omega_g)=\frac{\lambda_{\varphi_{g_1}}\lambda_{\varphi_{g_2}}}{\lambda_\varphi}\cdot
\frac{\widetilde{\varepsilon}_{g_1,g_2}}{\widetilde{\varepsilon}_{g_1',g_2'}}\cdot\omega_{g'},\quad 
\varphi^*(\omega_h)=\frac{\lambda_{\varphi_{h_1}}\lambda_{\varphi_{h_2}}}{\lambda_\varphi}\cdot
\frac{\widetilde{\varepsilon}_{h_1,h_2}}{\widetilde{\varepsilon}_{h_1',h_2'}}\cdot\omega_{h'},
\end{equation}
which implies 
\begin{equation}
\varphi^*(\overline v_g)=\frac{\lambda_{\varphi_{g_1}}\lambda_{\varphi_{g_2}}}{\lambda_\varphi^2}\cdot
\frac{\widetilde{\varepsilon}_{g_1',g_2'}}{\widetilde{\varepsilon}_{g_1,g_2}}\cdot\overline v_{g'},\quad 
\varphi^*(\overline v_h)=\frac{\lambda_{\varphi_{h_1}}\lambda_{\varphi_{h_2}}}{\lambda_\varphi^2}\cdot
\frac{\widetilde{\varepsilon}_{h_1',h_2'}}{\widetilde{\varepsilon}_{h_1,h_2}}\cdot\overline v_{h'}.
\end{equation}
\end{lemma}
\begin{proof}
Let $\T_{\CC^N}$ be the tangent sheaf on $\CC^N$.
For each $g''\in G_f$, define a poly-vector field $\widetilde\theta_{g''}\in\Gamma\left(\CC^N,\wedge^{N-N_{g''}}\T_{\CC^N}\right)$ by 
\[
\widetilde\theta_{g''}:=
\begin{cases}
\frac{\partial}{\partial x_{j_1}}\wedge \dots \wedge \frac{\partial}{\partial x_{j_{N-N_{g''}}}} & \text{if } I_{g''}^c=(j_1,\dots, j_{N-N_{g''}}),\  j_1<\dots <j_{N-N_{g''}}\\
1 & \text{if } I_{g''}^c=\emptyset
\end{cases}.
\]
Since we have $\varphi^\ast(\widetilde \omega_\id)=\lambda_{\varphi}\widetilde \omega_\id$ and 
$\varphi^\ast(\widetilde \omega_{g_i})=\lambda_{\varphi_{g_i}}\widetilde \omega_{g_i'}$ for $i=1,2$, 
the poly-vector field $\widetilde\theta_{g_i}$ transforms under $\varphi$ as 
\[
\widetilde\theta_{g_i}\mapsto \frac{\lambda_{\varphi_{g_i}}}{\lambda_\varphi}
\cdot \frac{\widetilde \varepsilon_{g_i}}{\widetilde \varepsilon_{g_i'}}\cdot \widetilde\theta_{g_i'},\quad i=1,2,
\]
where $\widetilde \varepsilon_{g_i}$ is the signature of the permutation $I_\id\longrightarrow I_{g_i}^c\sqcup I_{g_i}$
and $\widetilde \varepsilon_{g_i'}$ is the signature of the permutation $I_\id\longrightarrow I_{g_i'}^c\sqcup I_{g_i'}$.
Suppose that $\varphi^*(\omega_g)=\lambda_{\varphi_g}\omega_{g'}$ for some $\lambda_{\varphi_g}\in\CC^\ast$ 
and let $\widetilde \varepsilon_g$ be the signature of the permutation $I_\id\longrightarrow I_g^c\sqcup I_g$
and $\widetilde \varepsilon_{g'}$ be signature of the permutation $I_\id\longrightarrow I_{g'}^c\sqcup I_{g'}$.
Then,  $\widetilde\theta_g$ transforms under $\varphi$ as 
\[
\widetilde\theta_g\mapsto \frac{\lambda_{\varphi_g}}{\lambda_\varphi}\cdot 
\frac{\widetilde \varepsilon_{g}}{\widetilde \varepsilon_{g'}}\cdot \widetilde\theta_{g'},
\]
Note that $\widetilde\theta_g=\widetilde \varepsilon_{g_1,g_2}\widetilde\theta_{g_1}\wedge \widetilde\theta_{g_2}$ 
and $\widetilde\theta_{g'}=\widetilde \varepsilon_{g_1',g_2'}\widetilde\theta_{g_1'}\wedge \widetilde\theta_{g_2'}$.
Hence, we have 
\[
\frac{\lambda_{\varphi_g}}{\lambda_\varphi}\cdot
\frac{\widetilde \varepsilon_{g}\widetilde \varepsilon_{g_1',g_2'}}
{\widetilde \varepsilon_{g'}\widetilde \varepsilon_{g_1,g_2}}
=
\frac{\lambda_{\varphi_{g_1}}\lambda_{\varphi_{g_2}}}{\lambda_\varphi^2}\cdot \frac{\widetilde \varepsilon_{g_1}\widetilde \varepsilon_{g_2}}{\widetilde \varepsilon_{g_1'}\widetilde \varepsilon_{g_2'}}.
\]
Therefore, the statement is reduced to show that 
\[
\frac{\widetilde \varepsilon_{g_1}\widetilde \varepsilon_{g_2}}{\widetilde \varepsilon_{g}}
=
\frac{\widetilde \varepsilon_{g_1'}\widetilde \varepsilon_{g_2'}}{\widetilde \varepsilon_{g'}}.
\]
However, by calculating the number of elements less than $j$ in the sequences 
$I_{g_1}^c$, $I_{g_2}^c$ and $I_g^c$ for each element $j$ in $I_{g_1}^c$ or $I_{g_2}^c$,  
it turns out that the LHS of the above equation is equal to $(-1)^{(N-N_{g_1})(N-N_{g_2})}$.
Similarly, the RHS is equal to $(-1)^{(N-N_{g_1'})(N-N_{g_2'})}$. 
They coincide since we have $N_{g_1}=N_{g_1'}$ and $N_{g_2}=N_{g_2'}$.
\end{proof}
\begin{lemma}
We have 
\begin{equation}
\left[\varphi^*H_{g,h}\right]=\frac{\lambda_{\varphi_{g_2}}^2}{\lambda_{\varphi}^2}[H_{g',h'}].
\end{equation}
\end{lemma}
\begin{proof}
Recall Definition~\ref{definition: Hgh}, where 
$H_{g,h}$ is defined as a non-zero constant multiple of $\det \left(\frac{\partial^2 f}{\partial x_{i} \partial x_{j}}\right)_{i,j\in I_{g,h}}$.
Now, $I_{g,h}=I_{g_2}^c=I_{h_2}^c$, $I_{g',h'}=I_{g_2'}^c=I_{h_2'}^c$.
This is nothing but the transformation rule of the determinant under the automorphism $\varphi$.
\end{proof}
Since $g_2h_2=\id$ and $g_2'h_2'=\id$ by definition of the factorizations,
\[
N_{g_2}=N_{h_2}=N_{h_2'}=N_{g_2'},\quad \lambda_{\varphi_{g_2}}=\lambda_{\varphi_{h_2}},
\]
where we identify $\omega_{h_2}$ with $\omega_{g_2}$ under $\Omega_{f,h_2}=\Omega_{f,g_2}$
and $\omega_{h_2'}$ with $\omega_{g_2'}$ under $\Omega_{f,h_2'}=\Omega_{f,g_2'}$. 
By the above lemma, it follows that 
\begin{align*}
&\varphi^\ast(\overline v_g)\circ \varphi^\ast(\overline v_h)\\
=&\ \frac{\lambda_{\varphi_{g_1}}\lambda_{\varphi_{g_2}}\lambda_{\varphi_{h_1}}\lambda_{\varphi_{h_2}}}{\lambda_{\varphi}^4}\cdot
\frac{\widetilde{\varepsilon}_{g_1,g_2}}{\widetilde{\varepsilon}_{g_1',g_2'}}\cdot
\frac{\widetilde{\varepsilon}_{h_1,h_2}}{\widetilde{\varepsilon}_{h_1',h_2'}}\cdot \overline v_{g'}\circ \overline v_{h'}\\
=&\ \frac{\lambda_{\varphi_{g_1}}\lambda_{\varphi_{g_2}}\lambda_{\varphi_{h_1}}\lambda_{\varphi_{h_2}}}{\lambda_{\varphi}^4}\cdot
\frac{\widetilde{\varepsilon}_{g_1,g_2}}{\widetilde{\varepsilon}_{g_1',g_2'}}\cdot
\frac{\widetilde{\varepsilon}_{h_1,h_2}}{\widetilde{\varepsilon}_{h_1',h_2'}}\\
&\quad \cdot (-1)^{\frac{1}{2}(N-N_{g_2'})(N-N_{g_2'}-1)}\cdot \epi \left[-\frac{1}{2}\age(g_2')\right]\cdot 
\frac{\widetilde{\varepsilon}_{g_1',g_2'}\widetilde{\varepsilon}_{h_2',h_1'}}{\widetilde{\varepsilon}_{g_1',h_1'}}
\cdot\left[H_{g',h'}\right]\overline v_{g'h'}\\
=&\ (-1)^{\frac{1}{2}(N-N_{g_2})(N-N_{g_2}-1)}\cdot \epi \left[-\frac{1}{2}\age(g_2)\right]\cdot 
\frac{\widetilde{\varepsilon}_{g_1,g_2}\widetilde{\varepsilon}_{h_2,h_1}}{\widetilde{\varepsilon}_{g_1,h_1}}\\
&\quad \cdot\left(\frac{\lambda_{\varphi_{g_2}}^2}{\lambda_{\varphi}^2} \left[H_{g',h'}\right]\right)\left(\frac{\lambda_{\varphi_{g_1}}\lambda_{\varphi_{h_1}}}{\lambda_\varphi^2}\cdot 
\frac{\widetilde{\varepsilon}_{g_1,h_1}}{\widetilde{\varepsilon}_{g_1',h_1'}}\overline v_{g'h'}\right)\\
=&\ c_{g,h}\left[\varphi^\ast H_{g,h}\right]\varphi^\ast(\overline v_{gh})=\varphi^\ast(\overline v_g\circ \overline v_h),
\end{align*}
where we also used that 
\[
\widetilde{\varepsilon}_{h_1,h_2}=(-1)^{(N-N_{h_1})(N-N_{h_2})}\widetilde{\varepsilon}_{h_2,h_1},\quad
\widetilde{\varepsilon}_{h_1',h_2'}=(-1)^{(N-N_{h_1'})(N-N_{h_2'})}\widetilde{\varepsilon}_{h_2',h_1'}.
\]
Hence, we proved the algebra structure $\circ$ of $\A'$ is $\Aut(f,G)$-invariant.
The $G$-twisted $\ZZ/2\ZZ$-graded commutativity, axiom~\eqref{axiom: G-twisted comm}, is a direct consequence of Proposition~\ref{prop:Gtwistedcomm} 
since $H_{g,h}=H_{h,g}$ and $g^\ast(\overline v_h)=\epi[-\age(g_2)]\cdot \overline v_h$ which follows from the calculation
\begin{align*}
&g^\ast(\overline v_h)\vdash \zeta
=g^\ast(\overline v_h)\vdash \left(\epi\left[-\age(g)\right]g^\ast(\zeta)\right)
=\epi\left[-\age(g)\right]\cdot g^\ast(\overline \alpha_h \omega_h)\\
&\quad =\epi\left[-\age(g_2)\right]\cdot (\overline \alpha_h \omega_h)
=\left(\epi[-\age(g_2)]\cdot \overline v_h\right)\vdash \zeta.
\end{align*}
We have finished the proof of the proposition.
\end{proof}
We show the invariance of the bilinear form $J_{f,G}$ with respect to the product structure of $\A'$.
We use the notation in Definition~\ref{definition: Hgh}.
\begin{proposition}
For a spanning pair $(g,h)$ of elements of $G_f$, we have
\begin{multline}\label{eq:44}
J_{f,gh}\left(\overline v_g\vdash \omega_h,\left[\frac{1}{\mu_{f^{g\cap h}}}\hess(f^{g\cap h})\right]\omega_{(gh)^{-1}}\right)\\
=(-1)^{(N-N_g)(N-N_h)} 
J_{f,h}\left(\omega_h, \left((h^{-1})^\ast\overline v_g\right)\vdash\left(\left[\frac{1}{\mu_{f^{g\cap h}}}\hess(f^{g\cap h})\right]\omega_{(gh)^{-1}}\right)\right). 
\end{multline}
As a consequence, the algebra $\A'$ satisfies axiom~{\rm (v)}. 
\end{proposition}
\begin{proof}
Let $(g_1,g_2,h_1,h_2)$ be the factorization of the spanning pair $(g,h)$.
The LHS of the equation~\eqref{eq:44} is calculated as 
\begin{align*}
&\ J_{f,gh}\left(\overline v_g\vdash \omega_h,\left[\frac{1}{\mu_{f^{g\cap h}}}\hess(f^{g\cap h})\right]\omega_{(gh)^{-1}}\right) \\
=&\ \frac{1}{\overline\alpha_h}\cdot J_{f,gh}\left(\left(\overline v_g\circ \overline v_h\right)\vdash \zeta,\left[\frac{1}{\mu_{f^{g\cap h}}}\hess(f^{g\cap h})\right]\omega_{(gh)^{-1}}\right) \\
=&\ \frac{\overline\alpha_{gh}\overline c_{g,h}}{\overline\alpha_h}\cdot J_{f,gh}\left(\omega_{gh},\left[\frac{1}{\mu_{f^{g\cap h}}}\hess(f^{g\cap h})H_{g,h}\right]\omega_{(gh)^{-1}}\right) \\
=&\ \frac{\overline\alpha_{gh}}{\overline\alpha_h}\cdot (-1)^{\frac{1}{2}(N-N_{g_2})(N-N_{g_2}-1)}\cdot  \epi \left[-\frac{1}{2}\age(g_2)\right]\cdot \frac{\widetilde{\varepsilon}_{g_1,g_2}\widetilde{\varepsilon}_{h_2,h_1}}{\widetilde{\varepsilon}_{g_1,h_1}}\\
&\quad \cdot (-1)^{N-N_{gh}}\cdot \epi\left[-\frac{1}{2}\age(gh)\right]\cdot |G|\\
=&\ \frac{\overline\alpha_{gh}}{\overline\alpha_h}\cdot (-1)^{\frac{1}{2}(N-N_{g_2})(N-N_{g_2}-1)+(N-N_{gh})}\\
&\quad \cdot \epi \left[-\frac{1}{2}\age(g_1)-\frac{1}{2}\age(h_1)-\frac{1}{2}\age(g_2)\right]\cdot \frac{\widetilde{\varepsilon}_{g_1,g_2}\widetilde{\varepsilon}_{h_2,h_1}}{\widetilde{\varepsilon}_{g_1,h_1}}\cdot |G|.
\end{align*}
On the other hand, the RHS of the equation~\eqref{eq:44} is calculated as 
\begin{align*}
&\ (-1)^{(N-N_g)(N-N_h)} \cdot 
J_{f,h}\left(\omega_h, \left((h^{-1})^\ast\overline v_g\right)\vdash\left(\left[\frac{1}{\mu_{f^{g\cap h}}}\hess(f^{g\cap h})\right]\omega_{(gh)^{-1}}\right)\right)\\
=&\ \frac{1}{\overline\alpha_{(gh)^{-1}}}(-1)^{(N-N_g)(N-N_h)} \cdot \epi \left[-\age(h_2^{-1})\right]\\
&\quad \cdot J_{f,h}\left(\omega_h, \left(\left[\frac{1}{\mu_{f^{g\cap h}}}\hess(f^{g\cap h})\right]\overline v_g\circ \overline v_{(gh)^{-1}}\right)\vdash \zeta\right)\\
=&\ \frac{\overline\alpha_{h^{-1}}c_{g,(gh)^{-1}}}{\overline\alpha_{(gh)^{-1}}}(-1)^{(N-N_g)(N-N_h)} \cdot\epi \left[-\age(g_2)\right]\cdot 
J_{f,h}\left(\omega_h, \left[\frac{1}{\mu_{f^{g\cap h}}}\hess(f^{g\cap h})\right]\vdash \omega_{h^{-1}}\right)\\
=&\ \frac{\overline\alpha_{h^{-1}}}{\overline\alpha_{(gh)^{-1}}}(-1)^{(N-N_g)(N-N_h)+\frac{1}{2}(N-N_{g_1})(N-N_{g_1}-1)}
\cdot  \epi \left[-\frac{1}{2}\age(g_1)-\age(g_2)\right]\cdot \frac{\widetilde{\varepsilon}_{g_2,g_1}\widetilde{\varepsilon}_{g_1^{-1},h_1}}{\widetilde{\varepsilon}_{g_2,h_1}}\\
&\quad \cdot (-1)^{N-N_{h}}\cdot \epi\left[-\frac{1}{2}\age(h)\right]\cdot |G|\\
=&\ \frac{\overline\alpha_{h^{-1}}}{\overline\alpha_{(gh)^{-1}}}(-1)^{(N-N_g+1)(N-N_h)+\frac{1}{2}(N-N_{g_1})(N-N_{g_1}-1)-(N-N_{g_2})+(N-N_{g_1})(N-N_{g_2})}\\
&\quad\cdot  \epi \left[-\frac{1}{2}\age(g_1)-\frac{1}{2}\age(h_1)-\frac{1}{2}\age(g_2)\right]\cdot \frac{\widetilde{\varepsilon}_{g_1,g_2}\widetilde{\varepsilon}_{h_2,h_1}}{\widetilde{\varepsilon}_{g_1,h_1}}\cdot |G|,
\end{align*}
where we used that $\widetilde{\varepsilon}_{g_2^{-1},h_1}^{-1}=\widetilde{\varepsilon}_{g_2^{-1},h_1}=\widetilde{\varepsilon}_{h_2,h_1}$ and 
$\widetilde{\varepsilon}_{g_1^{-1},h_1}=\widetilde{\varepsilon}_{g_1,h_1}=\widetilde{\varepsilon}_{g_1,h_1}^{-1}$. 
We have $\overline\alpha_{gh}\overline\alpha_{(gh)^{-1}}=(-1)^{\frac{1}{2}(N-N_{gh})(N-N_{gh}+1)}$ and 
$\overline\alpha_{h}\overline\alpha_{h^{-1}}=(-1)^{\frac{1}{2}(N-N_{h})(N-N_{h}+1)}$ by \eqref{eq:aa}.
Hence, it follows from a direct calculation by the use of 
\begin{align*}
&N-N_{g}=(N-N_{g_1})+(N-N_{g_2}),\quad N-N_{h}=(N-N_{h_1})+(N-N_{h_2}), \\
&N-N_{gh}=N-N_{g_1g_2}=(N-N_{g_1})+(N-N_{h_1}),\quad N_{g_2}=N_{h_2},
\end{align*}
(cf. Proposition~\ref{prop: supergrading agrees}) that 
\begin{align*}
&\frac{1}{2}(N-N_{gh})(N-N_{gh}+1)+\frac{1}{2}(N-N_{g_2})(N-N_{g_2}-1)+(N-N_{gh})\\
&-\frac{1}{2}(N-N_{h})(N-N_{h}+1)+(N-N_g+1)(N-N_h)\\
&+\frac{1}{2}(N-N_{g_1})(N-N_{g_1}-1)-(N-N_{g_2})+(N-N_{g_1})(N-N_{g_2})\\
\equiv&\ 0\ ({\rm mod}\ 2),
\end{align*}
which gives the equation~\eqref{eq:44}
For $X\in \A'_g$, $\omega\in\Omega'_{f,h}$, $\omega'\in\Omega'_{f,G}$,
$J_{f,G}(X\vdash\omega,\omega')$ is non-zero only if $\omega'\in\Omega'_{f,(gh)^{-1}}$ and the pair $(g,h)$ is a spanning pair. 
Note that $I_g\cup I_h\cup I_{gh}=I_\id$ if and only if $I_h\cup I_{(gh)^{-1}}\cup I_{g^{-1}}=I_\id$,
which means the pair $(g,h)$ is a spanning pair if and only if the pair $(h,(gh)^{-1})$ is so.
Therefore, $J_{f,G}(X\vdash\omega,\omega')$ is non-zero if and only if $J_{f,G}(\omega,(h^{-1})^\ast X\vdash\omega')$ is so.
It follows that the axiom~(v) can be reduced to the equation~\eqref{eq:44}. 
\end{proof}
The last axiom (vi) is trivially satisfied for $\A'$.
Therefore, we have shown the existence of a $G_f$-twisted Jacobian algebra of $f$. 
Moreover, it is easy to see the following
\begin{proposition}
For each subgroup $G\subseteq G_f$, 
there exists a $G$-twisted Jacobian algebra of $f$.
\end{proposition}
\begin{proof}
Consider the subspace $\A'_G$ of $\A'$ defined by 
\[
\A'_G:=\bigoplus_{g\in G}\A'_g,
\]
the restriction of the product structure map $\circ:\A'\otimes_\CC\A'\longrightarrow \A'$ to 
$\A'_G\otimes_\CC \A'_G$ 
and the restriction of the $\A'$-module structure map $\vdash:\A'\otimes_\CC\Omega_{f,G_f}'\longrightarrow \A'$ to $\A'_G\otimes_\CC\Omega'_{F,G}$.
By the construction of these structures on $\A'$, it is almost obvious that they satisfy all the axioms in Definition~\ref{axioms}.
\end{proof}
We have finished the proof of Theorem~\ref{theorem_N}.
\section{Orbifold Jacobian algebras for ADE orbifolds}
The classification of invertible polynomials in three variables giving ADE singularities and 
the subgroups of their maximal diagonal symmetries preserving the holomorphic volume form is given as follows
(see also \cite{et} Section~8 Table~3).
\begin{table}[h]
{\small
\begin{center}
\begin{tabular}{c||l|c|c}
{\rm Type} & $f(x_1,x_2,x_3)$ & $G_f\cap {\rm SL}(3;\CC)$ & {\rm Singularity Type}\\
\hline
\hline
{\rm I} & $z_1^{2k+1}+z_2^2+z_3^2$,\quad $k \ge  1$ & $\left<\frac{1}{2}(0,1,1)\right>$ & $A_{2k}$\\
 & $z_1^{2k}+z_2^2+z_3^2$,\quad $k \ge  1$ & $\left<\frac{1}{2}(0,1,1),\frac{1}{2}(1,0,1)\right>$ & $A_{2k-1}$\\
 & $z_1^{3}+z_2^3+z_3^2$ & $\left<\frac{1}{3}(1,2,0)\right>$ & $D_4$ \\
 & $z_1^{4}+z_2^3+z_3^2$ & $\left<\frac{1}{2}(1,0,1)\right>$ & $E_6$ \\
 & $z_1^{5}+z_2^3+z_3^2$ & $\{1\}$ & $E_8$ \\ 
\hline
{\rm I\!I} 
 & $z_1^2+z_2^2+z_2z_3^{2k}$,\quad $k \ge  1$ & $\left<\frac{1}{2}(1,0,1)\right>$ & $A_{4k-1}$\\
 & $z_1^2+z_2^2+z_2z_3^{2k+1}$,\quad $k \ge  1$ & $\left<\frac{1}{2}(0,1,1)\right>$ & $A_{4k+1}$\\
 & $z_1^2+z_2^{k-1}+z_2z_3^2$,\quad $k \ge  4$ & $\left<\frac{1}{2}(1,0,1)\right>$ & $D_{k}$\\
 & $z_1^3+z_2^2+z_2z_3^2$ & $\{1\}$ & $E_6$\\
 & $z_1^2+z_2^3+z_2z_3^3$ & $\{1\}$ & $E_7$ \\
\hline
{\rm I\!I\!I}
 & $z_1^2+z_3z_2^2+z_2z_3^{k+1}$,\quad $k \ge  1$ & $\{1\}$ & $D_{2k+2}$\\
\hline
{\rm I\!V} & $z_1^k+z_1z_2+z_2z_3^l$,\quad $k, l \ge  2$ & $\{1\}$ & $A_{kl-1}$\\
    & $z_1^2+z_1z_2^k+z_2z_3^2$,\quad $k \ge  2$ & $\{1\}$ & $D_{2k+1}$ \\
\hline
{\rm V} & $z_1z_2+z_2^kz_3+z_3^lz_1$,\quad $k,l \ge  1$ & $\{1\}$ & $A_{kl}$\\
\hline
\end{tabular}
\end{center}
\smallskip
\caption{{\small Classification of invertible polynomials giving ADE singularities and the groups of their diagonal symmetries 
preserving the holomorphic volume form.}}\label{table1}
}
\end{table}
As it is explained in Section~8 in \cite{et}, one can describe explicitly the geometry of vanishing cycles for 
the holomorphic map $\widehat f:\widehat{\CC^3/G}\longrightarrow \CC$. 
Here, $\widehat{\CC^3/G}$ is
a crepant resolution of $\CC^3/G$ and $\widehat f$ is the convolution of the resolution map $\widehat{\CC^3/G}\longrightarrow \CC^3/G$ and the induced one $f:\CC^3/G\longrightarrow \CC$. 
Note that $\widehat{\CC^3/G}$ is covered by some charts all isomorphic to $\CC^3$.
When $G$ respects one coordinate we only need to look at the resolutions of $\CC^2$ given in \cite{BaKn}. 
For $G\cong\ZZ/2\ZZ$ acting $(z_i, z_j)\mapsto (-z_i,-z_j)$, we have $\CC^3/G\cong\CC\times\{z^2=xy\}\subset\CC^4$ by $x=z_i^2, y=z_j^2, z=z_iz_j$ and we have the two charts $ \CC^3\rightarrow \CC^4$;
\[
(t,u,v)\mapsto(t,u,uv^2,uv) \text{ and }
(t,u,v)\mapsto(t,u^2v,v,uv).
\]
For $G\cong\ZZ/3\ZZ=\left<\frac{1}{3}(1,2,0)\right>$, we have $\CC^3/G\cong\CC\times\{z^3=xy\}\subset\CC^4$ by $x=z_1^3, y=z_2^3, z=z_1z_2$ and we have the three charts $ \CC^3\rightarrow \CC^4$;
\[
(t,u,v)\mapsto(t,u,u^2v^3,uv) \text{ , }
(t,u,v)\mapsto(t,u^2v,uv^2,uv) \text{ and }
(t,u,v)\mapsto(t,u^3v^2,v,uv).
\]
We shall calculate the restriction of $\widehat f$ on each chart based on the classification in Table~\ref{table1}.
\noindent
{\bf 1.} For the pair
\begin{equation}
f:=z_1^{k+1}+z_2^2+z_3^2\ (k\ge 1),\quad G:=\left<\frac{1}{2}(0,1,1)\right>,
\end{equation}
we have in the two charts 
\begin{equation}
\widehat f(t,u,v)= t^{k+1}+u+uv^2 \text{ and } \widehat f(t,u,v)= t^{k+1}+u^2v+v.
\end{equation}
Critical points of $\widehat{f}$ are on the intersection of two charts.
\noindent
{\bf 2.} 
For the pair
\begin{equation}
f:=z_1^{2k}+z_2^2+z_3^2\ (k\ge 1),\quad G:=\left<\frac{1}{2}(1,0,1)\right>,
\end{equation}
we have in the two charts 
\begin{equation}
\widehat f(t,u,v)= t^2+u^k+uv^2 \text{ and } \widehat f(t,u,v)= t^2+u^2kv^k+v.
\end{equation}
Critical points of $\widehat{f}$ are on the first chart.
\noindent
{\bf 3.} 
For $k\ge 1$, set
\begin{equation}
f:=z_1^{2k}+z_2^2+z_3^2\ (k\ge 1),\quad G:=\left<\frac{1}{2}(0,1,1),\frac{1}{2}(1,0,1)\right>.
\end{equation}
Here, since the resulution is not unique, we take $A$-Hilb $\CC^3$ of \cite{CR} where $A=\ZZ/2\ZZ\times \ZZ/2\ZZ$. 
We have $\CC^3/G\cong\CC\times\{z^3=wxy\}\subset\CC^4$ by $w=z_1^2, x=z_2^2, y=z_3^2, z=z_1z_2z_3$ 
and we have four charts $ \CC^3\rightarrow \CC^4$;
\[
(t,u,v)\mapsto(t,u,tuv^2,tuv) \text{ , }
(t,u,v)\mapsto(t,tu^2v,v,tuv) \text{ , }
\]
\[
(t,u,v)\mapsto(t^2uv,u,v,tuv) \text{ and }
(t,u,v)\mapsto(tu,uv,tv,tuv) .
\]
Then we have in the four charts 
\begin{subequations}
\begin{equation}
\widehat f(t,u,v)= t^k+u+tuv^2 \text{ , } \widehat f(t,u,v)= t^k+tu^2v+v,
\end{equation}
\begin{equation}
\widehat f(t,u,v)= t^{2k}u^kv^k+u+v \text{ and } \widehat f(t,u,v)= t^ku^k+uv+tv.
\end{equation}
\end{subequations}
Critical points of $\widehat{f}$ are on the fourth chart.
\noindent
{\bf 4.} 
For the pair
\begin{equation}
f:=z_1^3+z_2^3+z_3^2,\quad G:=\left<\frac{1}{3}(1,2,0)\right>,
\end{equation}
we have in the three charts 
\begin{equation}
\widehat f(t,u,v)= t^2+u+u^2v^3,\ \widehat f(t,u,v)= t^2+u^2v+uv^2 \text{ and }\widehat f(t,u,v)= t^2+u^3v^2+v.
\end{equation}
Critical points of $\widehat{f}$ are on the second chart.
\noindent
{\bf 5.} 
For the pair
\begin{equation}
f:=z_1^4+z_2^3+z_3^2,\quad G:=\left<\frac{1}{2}(1,0,1)\right>,
\end{equation}
we have in the two charts 
\begin{equation}
\widehat f(t,u,v)= t^3+u^2+uv^2 \text{ and } \widehat f(t,u,v)= t^3+u^4v^2+v.
\end{equation}
Critical points of $\widehat{f}$ are on the first chart.
\noindent
{\bf 6.} 
For the pair
\begin{equation}
f:=z_1^2+z_2^2+z_2z_3^{2k}\ (k\ge 1),\quad G:=\left<\frac{1}{2}(1,0,1)\right>,
\end{equation}
we have in the two charts 
\begin{equation}
\widehat f(t,u,v)=t^2+t u^k v^{2k}+u \text{ and } \widehat f(t,u,v)= t^2+t v^k+v u^2.
\end{equation}
Critical points of $\widehat{f}$ are on the second chart.
\noindent
{\bf 7.} 
For the pair
\begin{equation}
f:=z_1^2+z_2^2+z_2z_3^{2k+1}\ (k\ge 1),\quad G:=\left<\frac{1}{2}(0,1,1)\right>,
\end{equation}
we have in the two charts 
\begin{equation}
\widehat f(t,u,v)=t^2+u+u^{k+1} v^{2k+1} \text{ and } \widehat f(t,u,v)= t^2+v u^2 +u v^{k+1}.
\end{equation}
Critical points of $\widehat{f}$ are on the second chart.
\noindent
{\bf 8.} 
For the pair
\begin{equation}
f:=z_1^2+z_2^{k-1}+z_2z_3^2\ (k\ge 4),\quad G:=\left<\frac{1}{2}(1,0,1)\right>,
\end{equation}
we have in the two charts 
\begin{equation}
\widehat f(t,u,v)=t^{k-1}+t u v^2+u \text{ and } \widehat f(t,u,v)= t^{k-1}+t v+v u^2.
\end{equation}
Critical points of $\widehat{f}$ are on the second chart.
To summarize, we observed that critical points of the map $\widehat f$ are contained in one chart 
isomorphic to $\CC^3$. The restriction of $\widehat f$ to the chart is given by
$\overline{f}$ defined in Table~\ref{table2}.
\begin{table}[h]
{\small
\begin{center}
\begin{tabular}{l||l|c||c}
&$f(x_1,x_2,x_3)$ & $G$ & $\overline{f}(y_1,y_2,y_3)$\\
\hline
1.&$z_1^{k+1}+z_2^2+z_3^2$,\quad $k \ge  1$ & $\left<\frac{1}{2}(0,1,1)\right>$ & $y_1^{k+1}+y_2+y_2y_3^2$\\
2.&$z_1^{2k}+z_2^2+z_3^2$,\quad $k \ge  1$ & $\left<\frac{1}{2}(1,0,1)\right>$ & $y_1^2+y_2^k+y_2y_3^2$\\
3.&$z_1^{2k}+z_2^2+z_3^2$,\quad $k \ge  1$ & $\left<\frac{1}{2}(0,1,1),\frac{1}{2}(1,0,1)\right>$ & $y_1^ky_2^k+y_1y_3+y_2y_3$ \\
4.&$z_1^{3}+z_2^3+z_3^2$ & $\left<\frac{1}{3}(1,2,0)\right>$ & $y_1^2+y_3y_2^2+y_2y_3^2$ \\
5.&$z_1^{4}+z_2^3+z_3^2$ & $\left<\frac{1}{2}(1,0,1)\right>$ & $y_1^3+y_2^2+y_2y_3^2$ \\
6.&$z_1^2+z_2^2+z_2z_3^{2k}$,\quad $k \ge  1$ & $\left<\frac{1}{2}(1,0,1)\right>$ & $y_1^2+y_1y_2^k+y_2y_3^2$\\
7.&$z_1^2+z_2^2+z_2z_3^{2k+1}$,\quad $k \ge  1$ & $\left<\frac{1}{2}(0,1,1)\right>$ & $y_1^2+y_3y_2^2+y_2y_3^{k+1}$\\
8.&$z_1^2+z_2^{k-1}+z_2z_3^2$,\quad $k \ge  4$ & $\left<\frac{1}{2}(1,0,1)\right>$ & $y_1^{k-1}+y_1y_2+y_2y_3^2$\\
\hline
\end{tabular}
\end{center}
}
\smallskip
\caption{$(f,G)\cong (\overline{f},\{1\})$}\label{table2}
\end{table}

Therefore, concerning the geometry of vanishing cycles, the pair $(f,G)$ is equivalent to the pair $(\overline{f},\{1\})$.
Then, it is quite natural to expect that the orbifold Jacobian algebra $\Jac(f,G)$ of $(f,G)$ is isomorphic 
to the one $\Jac(\overline{f},\{1\})$ of the pair $(\overline{f},\{1\})$, the usual Jacobian algebra $\Jac(\overline{f})$ of $\overline{f}$, which is the following 
\begin{theorem}\label{thm_ADEiso}
There is an isomorphism of Frobenius algebras 
\begin{equation}
\Jac(f,G)\cong \Jac(\overline{f},\{1\})
\end{equation}
for all $f$ and $\overline{f}$ in Table~\ref{table2}.
\end{theorem} 
\begin{proof}
We shall give a proof of this theorem based on the classification in Table~\ref{table2}.
Let the notation be as in Section~4. 
For each $g\in G$ let $K_g$ be the maximal subgroup of $G$ fixing $\Fix(g)$. 
Define $e_g\in \Jac(f,G)$ by $e_g:=\frac{1}{|K_g|}v_g$, which is more natural element than $v_g$.
\noindent
{\bf 1.} For $k\ge 1$, set
\begin{equation}
f:=z_1^{k+1}+z_2^2+z_3^2,\quad G:=\left<g\right>,\ g:=\frac{1}{2}(0,1,1),
\end{equation}
\begin{equation}
\overline{f}:=y_1^{k+1}+y_2+y_2y_3^2.
\end{equation}
The Jacobian algebra $\Jac(f,\{1\})$ and the bilinear form $J_{f,\{1\}}$ on $\Omega_{f,\{1\}}$ can be calculated as
\begin{equation}
\Jac(f,\{1\})=\CC[z_1,z_2,z_3]\left/\left((k+1)z_1^{k},2z_2,2z_3\right)\right.\cong\left<[1],[z_1],\dots, [z_1]^{k-1}\right>_\CC,
\end{equation}
\begin{equation}
J_{f,\{1\}}\left([dz_1\wedge dz_2\wedge dz_3],[z_1^{k-1}dz_1\wedge dz_2\wedge dz_3]\right)=\frac{1}{4(k+1)}.
\end{equation}
As a $\CC$-module, the orbifold Jacobian algebra $\Jac(f,G)$ is of the following form$:$
\begin{equation}
\Jac(f,G)\cong\left<e_\id,[z_1],\dots, [z_1]^{k-1}\right>_\CC\oplus \left<e_g,[z_1] e_g,\dots, [z_1]^{k-1} e_g\right>_\CC.
\end{equation}
Note that $\dim_\CC \Jac(f,G)=2k$.
The bilinear form $J_{f,G}$ on $\Omega_{f,G}$ can be calculated as
\begin{subequations}
\begin{eqnarray}
J_{f,\id}\left(e_{\id}\vdash\zeta,[z_1]^{k-1}\vdash\zeta\right)&=&
J_{f,\id}\left([dz_1\wedge dz_2\wedge dz_3],[z_1^{k-1}dz_1\wedge dz_2\wedge dz_3]\right)\\
&=& 2\cdot\frac{1}{4(k+1)}=\frac{1}{2(k+1)},\\
J_{f,g}\left(e_{g}\vdash\zeta,[z_1]^{k-1}e_{g}\vdash\zeta\right)&=&\frac{1}{4}J_{f,g}\left([dz_1],[z_1^{k-1}dz_1]\right)\\
&=&\frac{1}{4}\cdot(-1)\cdot{2}\cdot\frac{1}{k+1}=-\frac{1}{2(k+1)},
\end{eqnarray}
\end{subequations}
which imply the following relations in the orbifold Jacobian algebra $\Jac(f,G):$
\begin{equation}
[z_1]^k=0,\quad e_g^2=-e_\id.
\end{equation}
On the other hand, the Jacobian algebra $\Jac(\overline{f},\{1\})$ is given by
\begin{eqnarray}
\Jac(\overline{f},\{1\})&=&\CC[y_1,y_2,y_3]\left/\left((k+1)y_1^{k},1+y_3^2,2y_2y_3\right)\right.\\
&\cong& \CC[y_1,y_3]\left/\left(y_1^{k},y_3^2+1\right)\right..
\end{eqnarray}
Note that $\dim_\CC \Jac(\overline{f},\{1\})=2k$.
Therefore, we have an algebra isomorphism
\begin{equation}
\Jac(\overline{f},\{1\})\stackrel{\cong}{\longrightarrow} \Jac(f,G),\quad [y_1]\mapsto [z_1],\ [y_3]\mapsto e_g,
\end{equation}
which is, moreover, an isomorphism of Frobenius algebras since we have
\begin{equation}
J_{\overline{f},\{1\}}\left([dy_1\wedge dy_2\wedge dy_3],[y_1^{k-1}dy_1\wedge dy_2\wedge dy_3]\right)=\frac{1}{2(k+1)}.
\end{equation}
\noindent
{\bf 2.} For $k\ge 1$, set
\begin{equation}
f:=z_1^{2k}+z_2^2+z_3^2,\quad G:=\left<g\right>,\ g:=\frac{1}{2}(1,0,1),
\end{equation}
\begin{equation}
\overline{f}:=y_1^2+y_2^k+y_2y_3^2.
\end{equation}
The Jacobian algebra $\Jac(f,\{1\})$ and the bilinear form $J_{f,\{1\}}$ on $\Omega_{f,\{1\}}$ can be calculated as
\begin{equation}
\Jac(f,\{1\})=\CC[z_1,z_2,z_3]\left/\left(2k z_1^{2k-1},2z_2,2z_3\right)\right.\cong\left<[1],[z_1],\dots, [z_1]^{2k-2}\right>_\CC,
\end{equation}
\begin{equation}
J_{f,\{1\}}\left([dz_1\wedge dz_2\wedge dz_3],[z_1^{2k-2}dz_1\wedge dz_2\wedge dz_3]\right)=\frac{1}{8k}.
\end{equation}
As a $\CC$-module, the orbifold Jacobian algebra $\Jac(f,G)$ is of the following form$:$
\begin{equation}
\Jac(f,G)\cong\left<e_\id,[z_1^2],\dots, [z_1^2]^{k-1}\right>_\CC\oplus \left<e_g\right>_\CC.
\end{equation}
Note that $\dim_\CC \Jac(f,G)=k+1$.
The bilinear form $J_{f,G}$ on $\Omega_{f,G}$ can be calculated as
\begin{subequations}
\begin{eqnarray}
J_{f,\id}\left(e_{\id}\vdash\zeta,[z_1^2]^{k-1}\vdash\zeta\right)&=&
J_{f,\id}\left([dz_1\wedge dz_2\wedge dz_3],[z_1^{2k-2}dz_1\wedge dz_2\wedge dz_3]\right)\\
&=& 2\cdot\frac{1}{8k}=\frac{1}{4k},\\
J_{f,g}\left(e_{g}\vdash\zeta,e_{g}\vdash\zeta\right)&=&\frac{1}{4}J_{f,g}\left([dz_2],[dz_2]\right)\\
&=&\frac{1}{4}\cdot(-1)\cdot{2}\cdot\frac{1}{2}=-\frac{1}{4},
\end{eqnarray}
\end{subequations}
which imply the following relations in the orbifold Jacobian algebra $\Jac(f,G):$
\begin{equation}
[z_1^2]\circ e_g=0,\quad e_g^2=-k[z_1^2]^{k-1}.
\end{equation}
On the other hand, the Jacobian algebra $\Jac(\overline{f},\{1\})$ is given by
\begin{eqnarray}
\Jac(\overline{f},\{1\})&=&\CC[y_1,y_2,y_3]\left/\left(2y_1,ky_2^{k-1}+y_3^2,2y_2y_3\right)\right.\\
&\cong&\CC[y_2,y_3]\left/\left(ky_2^{k-1}+y_3^2,y_2y_3\right)\right..
\end{eqnarray}
Note that $\dim_\CC \Jac(\overline{f},\{1\})=k+1$.
Therefore, we have an algebra isomorphism
\begin{equation}
\Jac(\overline{f},\{1\})\stackrel{\cong}{\longrightarrow} \Jac(f,G),
\quad [y_2]\mapsto [z_1^2],\ [y_3]\mapsto e_g,
\end{equation}
which is, moreover, an isomorphism of Frobenius algebras since we have
\begin{equation}
J_{\overline{f},\{1\}}\left([dy_1\wedge dy_2\wedge dy_3],[y_2^{k-1}dy_1\wedge dy_2\wedge dy_3]\right)=\frac{1}{4k}.
\end{equation}
\noindent
{\bf 3.} For $k\ge 1$, set
\begin{equation}
f:=z_1^{2k}+z_2^2+z_3^2,\quad G:=\left<g,h\right>,\ g:=\frac{1}{2}(0,1,1),\ h:=\frac{1}{2}(1,0,1),
\end{equation}
\begin{equation}
\overline{f}:=y_1^ky_2^k+y_1y_3+y_2y_3.
\end{equation}
The Jacobian algebra $\Jac(f,\{1\})$ and the bilinear form $J_{f,\{1\}}$ on $\Omega_{f,\{1\}}$ can be calculated as
\begin{equation}
\Jac(f,\{1\})=\CC[z_1,z_2,z_3]\left/\left(2kz_1^{2k-1},2z_2,2z_3\right)\right.\cong\left<[1],[z_1],\dots, [z_1]^{2k-2}\right>_\CC,
\end{equation}
\begin{equation}
J_{f,\{1\}}\left([dz_1\wedge dz_2\wedge dz_3],[z_1^{2k-2}dz_1\wedge dz_2\wedge dz_3]\right)=\frac{1}{8k}.
\end{equation}
As a $\CC$-module, the orbifold Jacobian algebra $\Jac(f,G)$ is of the following form$:$
\begin{equation}
\Jac(f,G)\cong\left<e_\id,[z_1^2],\dots, [z_1^2]^{k-1}\right>_\CC
\oplus \left<e'_g, [z_1^2]e'_g,\dots, [z_1^2]^{k-2} e'_g\right>_\CC,
\end{equation}
where $e'_g:=[z_1] e_g$ since $\Jac(f,h)=\{0\}$ and $\Jac(f,gh)=\{0\}$. 
Note that $\dim_\CC \Jac(f,G)=2k-1$.
The bilinear form $J_{f,G}$ on $\Omega_{f,G}$ can be calculated as
\begin{subequations}
\begin{eqnarray}
J_{f,\id}\left(e_{\id}\vdash\zeta,[z_1^2]^{k-1}\vdash\zeta\right)&=&
J_{f,\id}\left([dz_1\wedge dz_2\wedge dz_3],[z_1^{2k-2}dz_1\wedge dz_2\wedge dz_3]\right)\\
&=& 4\cdot\frac{1}{8k}=\frac{1}{2k},\\
J_{f,g}\left(e'_{g}\vdash\zeta,[z_1^2]^{k-2} e'_g\vdash\zeta\right)&=&
\frac{1}{4}J_{f,g}\left([z_1dz_1],[z_1^{2k-3}dz_1]\right)\\
&=&\frac{1}{4}\cdot(-1)\cdot{4}\cdot\frac{1}{2k}=-\frac{1}{2k},
\end{eqnarray}
\end{subequations}
which imply the following relations in the orbifold Jacobian algebra $\Jac(f,G):$
\begin{equation}
[z_1^2]^{k-1}\circ e'_g=0,\quad (e'_g)^2=-[z_1^2].
\end{equation}
On the other hand, the Jacobian algebra $\Jac(\overline{f},\{1\})$
is given by
\begin{equation}
\Jac(\overline{f},\{1\})=\CC[y_1,y_2,y_3]\left/\left(ky_1^{k-1}y_2^k+y_3,ky_1^ky_2^{k-1}+y_3,y_1+y_2\right)\right..
\end{equation}
Note that $\dim_\CC \Jac(\overline{f},\{1\})=2k-1$.
Therefore, we have an algebra isomorphism
\begin{equation}
\Jac(\overline{f},\{1\})\stackrel{\cong}{\longrightarrow} \Jac(f,G),
\quad [y_1y_2]\mapsto [z_1^2],
\ [y_1]\mapsto e'_g,
\end{equation}
which is, moreover, an isomorphism of Frobenius algebras since we have
\begin{equation}
J_{\overline{f},\{1\}}\left([dy_1\wedge dy_2\wedge dy_3],[y_1^{k-1}y_2^{k-1}dy_1\wedge dy_2\wedge dy_3]\right)=\frac{1}{2k}.
\end{equation}
\noindent
{\bf 4.} Set
\begin{equation}
f:=z_1^3+z_2^3+z_3^2,\quad G:=\left<g\right>,\ g:=\frac{1}{3}(1,2,0),
\end{equation}
\begin{equation}
\overline{f}:=y_1^2+y_3y_2^2+y_2y_3^2.
\end{equation}
The Jacobian algebra $\Jac(f,\{1\})$ and the bilinear form $J_{f,\{1\}}$ on $\Omega_{f,\{1\}}$ can be calculated as
\begin{equation}
\Jac(f,\{1\})\cong\CC[z_1,z_2,z_3]\left/\left(3z_1^2,3z_2^2,2z_3\right)\right.
\cong\left<[1],[z_1],[z_2],[z_1z_2]\right>_\CC,
\end{equation}
\begin{equation}
J_{f,\{1\}}\left([dz_1\wedge dz_2\wedge dz_3],[z_1z_2dz_1\wedge dz_2\wedge dz_3]\right)=\frac{1}{18}.
\end{equation}
As a $\CC$-module, the orbifold Jacobian algebra $\Jac(f,G)$ is of the following form$:$
\begin{equation}
\Jac(f,G)\cong\left<e_\id,[z_1z_2]\right>_\CC\oplus \left<e_g,e_{g^{-1}}\right>_\CC.
\end{equation}
Note that $\dim_\CC \Jac(f,G)=4$.
The bilinear form $J_{f,G}$ on $\Omega_{f,G}$ can be calculated as
\begin{subequations}
\begin{eqnarray}
J_{f,\id}\left(e_{\id}\vdash\zeta,[z_1z_2]\vdash\zeta\right)&=&
J_{f,\id}\left([dz_1\wedge dz_2\wedge dz_3],[z_1z_2dz_1\wedge dz_2\wedge dz_3]\right)\\
&=& 3\cdot\frac{1}{18}=\frac{1}{6},\\
J_{f,g}\left(e_{g}\vdash\zeta,e_{g^{-1}}\vdash\zeta\right)&=&
\frac{1}{9}J_{f,g}\left([dz_3],[dz_3]\right)\\
&=&\frac{1}{9}\cdot(-1)\cdot{3}\cdot\frac{1}{2}=-\frac{1}{6},
\end{eqnarray}
\end{subequations}
which imply the following relations in the orbifold Jacobian algebra $\Jac(f,G):$
\begin{equation}
e_g^2=0,\quad e_{g^{-1}}^2=0,\quad 
e_g\circ e_{g^{-1}}=-[z_1z_2].
\end{equation}
On the other hand, the Jacobian algebra $\Jac(\overline{f},\{1\})$
is given by
\begin{eqnarray}
\Jac(\overline{f},\{1\})&=&\CC[y_1,y_2,y_3]\left/\left(2y_1,2y_3y_2+y_3^2,y_2^2+2y_2y_3\right)\right.\\
&\cong&\CC[y_2,y_3]\left/\left(2y_3y_2+y_3^2,y_2^2+2y_2y_3\right)\right..
\end{eqnarray}
Note that $\dim_\CC \Jac(\overline{f},\{1\})=4$.
Therefore, we have an algebra isomorphism
\begin{multline}
\Jac(\overline{f},\{1\})\stackrel{\cong}{\longrightarrow} \Jac(f,G),\\
[y_2]\mapsto e^{\frac{2\pi\sqrt{-1}}{3}} e_g+e^{\frac{4\pi\sqrt{-1}}{3}} e_{g^{-1}},\ [y_3]\mapsto e^{\frac{4\pi\sqrt{-1}}{3}} e_g+e^{\frac{2\pi\sqrt{-1}}{3}} e_{g^{-1}},
\end{multline}
which is, moreover, an isomorphism of Frobenius algebras since we have
\begin{equation}
J_{\overline{f},\{1\}}\left([dy_1\wedge dy_2\wedge dy_3],[y_2y_3dy_1\wedge dy_2\wedge dy_3]\right)=\frac{1}{6}.
\end{equation}
\noindent
{\bf 5.} Set
\begin{equation}
f:=z_1^4+z_2^3+z_3^2,\quad G:=\left<g\right>,\ g:=\frac{1}{2}(1,0,1),
\end{equation}
\begin{equation}
\overline{f}:=y_1^3+y_2^2+y_2y_3^2.
\end{equation}
The Jacobian algebra $\Jac(f,\{1\})$ and the bilinear form $J_{f,\{1\}}$ on $\Omega_{f,\{1\}}$ can be calculated as
\begin{equation}
\Jac(f,\{1\})=\CC[z_1,z_2,z_3]\left/\left(4z_1^3,3z_2^2,2z_3\right)\right.\cong\left<[1],[z_1],[z_2],[z_1^2],[z_1z_2],[z_1^2z_2]\right>_\CC,
\end{equation}
\begin{equation}
J_{f,\{1\}}\left([dz_1\wedge dz_2\wedge dz_3],[z_1^{2}z_2dz_1\wedge dz_2\wedge dz_3]\right)=\frac{1}{24}.
\end{equation}
As a $\CC$-module, the orbifold Jacobian algebra $\Jac(f,G)$ is of the following form$:$
\begin{equation}
\Jac(f,G)\cong\left<e_\id,[z_2],[z_1^2], [z_1^2z_2]\right>_\CC\oplus \left<e_g, [z_2] e_g\right>_\CC.
\end{equation}
Note that $\dim_\CC \Jac(f,G)=6$.
The bilinear form $J_{f,G}$ on $\Omega_{f,G}$ can be calculated as
\begin{subequations}
\begin{eqnarray}
J_{f,\id}\left(e_{\id}\vdash\zeta,[z_1^2z_2]\vdash\zeta\right)&=&
J_{f,\id}\left([dz_1\wedge dz_2\wedge dz_3],[z_1^{2}z_2dz_1\wedge dz_2\wedge dz_3]\right)\\
&=& 2\cdot\frac{1}{24}=\frac{1}{12},\\
J_{f,g}\left(e_{g}\vdash\zeta,[z_2] e_g\vdash\zeta\right)&=&
\frac{1}{4}J_{f,g}\left([dz_2],[z_2dz_2]\right)\\
&=&\frac{1}{4}\cdot(-1)\cdot{2}\cdot\frac{1}{3}=-\frac{1}{6},
\end{eqnarray}
\end{subequations}
which imply the following relations in the orbifold Jacobian algebra $\Jac(f,G):$
\begin{equation}
[z_2]^2=0,\quad 
[z_1^2]\circ e_g=0,\quad e_g^2=-2[z_1^2].
\end{equation}
On the other hand, the Jacobian algebra $\Jac(\overline{f},\{1\})$
is given by
\begin{equation}
\Jac(\overline{f},\{1\})=\CC[y_1,y_2,y_3]\left/\left(3y_1^2,2y_2+y_3^2,2y_2y_3\right)\right..
\end{equation}
Note that $\dim_\CC \Jac(\overline{f},\{1\})=6$.
Therefore, we have an algebra isomorphism
\begin{equation}
\Jac(\overline{f},\{1\})\stackrel{\cong}{\longrightarrow} \Jac(f,G),
\quad [y_1]\mapsto [z_2],\ [y_2]\mapsto [z_1^2],\ [y_3]\mapsto e_g,
\end{equation}
which is, moreover, an isomorphism of Frobenius algebras since we have
\begin{equation}
J_{\overline{f},\{1\}}\left([dy_1\wedge dy_2\wedge dy_3],[y_1y_2dy_1\wedge dy_2\wedge dy_3]\right)=\frac{1}{12}.
\end{equation}
\noindent
{\bf 6.} For $k\ge 1$, set
\begin{equation}
f:=z_1^2+z_2^2+z_2z_3^{2k},\quad G:=\left<g\right>,\ g:=\frac{1}{2}(1,0,1),
\end{equation}
\begin{equation}
\overline{f}:=y_1^2+y_1y_2^k+y_2y_3^2.
\end{equation}
The Jacobian algebra $\Jac(f,\{1\})$ and the bilinear form $J_{f,\{1\}}$ on $\Omega_{f,\{1\}}$ can be calculated as
\begin{eqnarray}
\Jac(f,\{1\})&=&\CC[z_1,z_2,z_3]\left/\left(2z_1,2z_2+z_3^{2k},2kz_2z_3^{2k-1}\right)\right.\\
&\cong&\left<[1],[z_3],\dots,[z_3]^{2k-1},[z_2],[z_2z_3],\dots, [z_2z_3^{2k-2}]\right>_\CC,
\end{eqnarray}
\begin{equation}
J_{f,\{1\}}\left([dz_1\wedge dz_2\wedge dz_3],[z_2z_3^{2k-2}dz_1\wedge dz_2\wedge dz_3]\right)=\frac{1}{8k}.
\end{equation}
As a $\CC$-module, the orbifold Jacobian algebra $\Jac(f,G)$ is of the following form$:$
\begin{equation}
\Jac(f,G)\cong\left<e_\id,[z_3^2],\dots, [z_3^2]^{k-1}, [z_2],[z_2]\cdot [z_3^2],\dots, [z_2]\cdot [z_3^2]^{k-1}\right>_\CC\oplus \left<e_g\right>_\CC.
\end{equation}
Note that $\dim_\CC \Jac(f,G)=2k+1$.
The bilinear form $J_{f,G}$ on $\Omega_{f,G}$ can be calculated as
\begin{subequations}
\begin{eqnarray}
J_{f,\id}\left(e_{\id}\vdash\zeta,[z_2]\cdot [z_3^2]^{k-1}\vdash\zeta\right)&=&
J_{f,\id}\left([dz_1\wedge dz_2\wedge dz_3],[z_2z_3^{2k-2}dz_1\wedge dz_2\wedge dz_3]\right)\\
&=& 2\cdot\frac{1}{8k}=\frac{1}{4k},\\
J_{f,g}\left(e_{g}\vdash\zeta,e_g\vdash\zeta\right)&=&
\frac{1}{4}J_{f,g}\left([dz_2],[dz_2]\right)\\
&=&\frac{1}{4}\cdot(-1)\cdot{2}\cdot\frac{1}{2}=-\frac{1}{4},
\end{eqnarray}
\end{subequations}
which imply the following relations in the orbifold Jacobian algebra $\Jac(f,G):$
\begin{equation}
2[z_2]+[z_3^2]^{k}=0,\quad 
[z_3^2]\circ e_g=0,\quad e_g^2=-k[z_2]\cdot [z_3^2]^{k-1}.
\end{equation}
On the other hand, the Jacobian algebra $\Jac(\overline{f},\{1\})$
is given by
\begin{equation}
\Jac(\overline{f},\{1\})=\CC[y_1,y_2,y_3]\left/\left(2y_1+y_2^k,ky_1y_2^{k-1}+y_3^2,2y_2y_3\right)\right..
\end{equation}
Note that $\dim_\CC \Jac(\overline{f},\{1\})=2k+1$.
Therefore, we have an algebra isomorphism
\begin{equation}
\Jac(\overline{f},\{1\})\stackrel{\cong}{\longrightarrow} \Jac(f,G),
\quad [y_1]\mapsto [z_2],\ [y_2]\mapsto [z_3^2],\ [y_3]\mapsto e_g,
\end{equation}
which is, moreover, an isomorphism of Frobenius algebras since we have
\begin{equation}
J_{\overline{f},\{1\}}\left([dy_1\wedge dy_2\wedge dy_3],[y_1y_2^{k-1}dy_1\wedge dy_2\wedge dy_3]\right)=\frac{1}{4k}.
\end{equation}
\noindent
{\bf 7.} For $k\ge 1$, set
\begin{equation}
f:=z_1^2+z_2^2+z_2z_3^{2k+1},\quad G:=\left<g\right>,\ g:=\frac{1}{2}(0,1,1),
\end{equation}
\begin{equation}
\overline{f}:=y_1^2+y_3y_2^2+y_2y_3^{k+1}.
\end{equation}
The Jacobian algebra $\Jac(f,\{1\})$ and the bilinear form $J_{f,\{1\}}$ on $\Omega_{f,\{1\}}$ can be calculated as
\begin{eqnarray}
\Jac(f,\{1\})&=&\CC[z_1,z_2,z_3]\left/\left(2z_1,2z_2+z_3^{2k+1},(2k+1)z_2z_3^{2k}\right)\right.\\
&\cong&\left<[1],[z_3],\dots,[z_3]^{2k},[z_2],[z_2z_3],\dots, [z_2z_3^{2k-1}]\right>_\CC,
\end{eqnarray}
\begin{equation}
J_{f,\{1\}}\left([dz_1\wedge dz_2\wedge dz_3],[z_2z_3^{2k-1}dz_1\wedge dz_2\wedge dz_3]\right)=\frac{1}{4(2k+1)}.
\end{equation}
As a $\CC$-module, the orbifold Jacobian algebra $\Jac(f,G)$ is of the following form$:$
\begin{equation}
\Jac(f,G)\cong\left<e_\id,[z_3^2],\dots, [z_3^2]^{k}, [z_2z_3],[z_2z_3]\cdot [z_3^2],\dots, [z_2z_3]\cdot [z_3^2]^{k-1}\right>_\CC\oplus \left<e_g\right>_\CC.
\end{equation}
Note that $\dim_\CC \Jac(f,G)=2(k+1)$.
The bilinear form $J_{f,G}$ on $\Omega_{f,G}$ can be calculated as
\begin{subequations}
\begin{eqnarray}
J_{f,\id}\left(e_{\id}\vdash\zeta,[z_2z_3]\cdot [z_3^2]^{k-1}\vdash\zeta\right)&=&
J_{f,\id}\left([dz_1\wedge dz_2\wedge dz_3],[z_2z_3^{2k-1}dz_1\wedge dz_2\wedge dz_3]\right)\\
&=& 2\cdot\frac{1}{4(2k+1)}=\frac{1}{2(2k+1)},\\
J_{f,g}\left(e_{g}\vdash\zeta,e_g\vdash\zeta\right)&=&
\frac{1}{4}J_{f,g}\left([dz_1],[dz_1]\right)\\
&=&\frac{1}{4}\cdot(-1)\cdot{2}\cdot\frac{1}{2}=-\frac{1}{4},
\end{eqnarray}
\end{subequations}

which imply the following relations in the orbifold Jacobian algebra $\Jac(f,G):$
\begin{equation}
2[z_2z_3]+[z_3^2]^{k+1}=0,\quad 
[z_3^2]\circ e_g=0,\quad e_g^2=-\frac{2k+1}{2}[z_2z_3] [z_3^2]^{k-1}.
\end{equation}
On the other hand, the Jacobian algebra $\Jac(\overline{f},\{1\})$
is given by
\begin{eqnarray}
\Jac(\overline{f},\{1\})&=&\CC[y_1,y_2,y_3]\left/\left(2y_1,2y_3y_2+y_3^{k+1},y_2^2+(k+1)y_2y_3^k\right)\right.\\
&\cong&\CC[y_2,y_3]\left/\left(2y_3y_2+y_3^{k+1},y_2^2+(k+1)y_2y_3^k\right)\right..
\end{eqnarray}
Note that $\dim_\CC \Jac(\overline{f},\{1\})=2(k+1)$.
Therefore, we have an algebra isomorphism
\begin{equation}
\Jac(\overline{f},\{1\})\stackrel{\cong}{\longrightarrow} \Jac(f,G),\quad
[y_2]\mapsto e_g-\frac{1}{2}[z_3^2]^k,\ [y_3]\mapsto [z_3^2],
\end{equation}
which is, moreover, an isomorphism of Frobenius algebras since we have
\begin{equation}
J_{\overline{f},\{1\}}\left([dy_1\wedge dy_2\wedge dy_3],[y_2y_3^{k}dy_1\wedge dy_2\wedge dy_3]\right)=\frac{1}{2(2k+1)}.
\end{equation}
\noindent
{\bf 8.} For $k\ge 4$, set
\begin{equation}
f:=z_1^2+z_2^{k-1}+z_2z_3^2,\quad G:=\left<g\right>,\ g:=\frac{1}{2}(1,0,1),
\end{equation}
\begin{equation}
\overline{f}:=y_1^{k-1}+y_1y_2+y_2y_3^2.
\end{equation}
The Jacobian algebra $\Jac(f,\{1\})$ and the bilinear form $J_{f,\{1\}}$ on $\Omega_{f,\{1\}}$ can be calculated as
\begin{eqnarray}
\Jac(f,\{1\})&=&\CC[z_1,z_2,z_3]\left/\left(2z_1,(k-1)z_2^{k-2}+z_3^2,2z_2z_3\right)\right.\\
&\cong&\left<[1],[z_2],\dots,[z_2]^{k-2},[z_3]\right>_\CC,
\end{eqnarray}
\begin{equation}
J_{f,\{1\}}\left([dz_1\wedge dz_2\wedge dz_3],[z_2^{k-2}dz_1\wedge dz_2\wedge dz_3]\right)=\frac{1}{4(k-1)}.
\end{equation}
As a $\CC$-module, the orbifold Jacobian algebra $\Jac(f,G)$ is of the following form$:$
\begin{equation}
\Jac(f,G)\cong\left<e_\id,[z_2],\dots,[z_2]^{k-2}\right>_\CC\oplus \left<e_g, [z_2]e_g,\dots,[z_2]^{k-3}e_g\right>_\CC.
\end{equation}
Note that $\dim_\CC \Jac(f,G)=2k-3$.
The bilinear form $J_{f,G}$ on $\Omega_{f,G}$ can be calculated as
\begin{subequations}
\begin{eqnarray}
J_{f,\id}\left(e_{\id}\vdash\zeta,[z_2]^{k-2}\vdash\zeta\right)&=&
J_{f,\id}\left([dz_1\wedge dz_2\wedge dz_3],[z_2^{k-2}dz_1\wedge dz_2\wedge dz_3]\right)\\
&=& 2\cdot\frac{1}{4(k-1)}=\frac{1}{2(k-1)},\\
J_{f,g}\left(e_{g}\vdash\zeta,[z_2]^{k-3}e_g\vdash\zeta\right)&=&
\frac{1}{4}J_{f,g}\left([dz_2],[z_2^{k-3}dz_2]\right)\\
&=&\frac{1}{4}\cdot(-1)\cdot{2}\cdot\frac{1}{k-1}=-\frac{1}{2(k-1)},
\end{eqnarray}
\end{subequations}
which imply the following relations in the orbifold Jacobian algebra $\Jac(f,G):$
\begin{equation}
[z_2]^{k-2} \circ e_g=0,\quad e_g^2=-[z_2].
\end{equation}
On the other hand, the Jacobian algebra $\Jac(\overline{f},\{1\})$
is given by
\begin{equation}
\Jac(\overline{f},\{1\})=\CC[y_1,y_2,y_3]\left/\left((k-1)y_1^{k-2}+y_2,y_1+y_3^2,2y_2y_3\right)\right..
\end{equation}
Note that $\dim_\CC \Jac(\overline{f},\{1\})=2k-3$.
Therefore, we have an algebra isomorphism
\begin{equation}
\Jac(\overline{f},\{1\})\stackrel{\cong}{\longrightarrow} \Jac(f,G),
\quad [y_1]\mapsto [z_2],\ [y_2]\mapsto -(k-1)[z_2]^{k-2},\ [y_3]\mapsto e_g,
\end{equation}
which is, moreover, an isomorphism of Frobenius algebras since we have
\begin{equation}
J_{\overline{f},\{1\}}\left([dy_1\wedge dy_2\wedge dy_3],[y_1^{k-2}dy_1\wedge dy_2\wedge dy_3]\right)=\frac{1}{2(k-1)}.
\end{equation}
We finished the proof of Theorem~\ref{thm_ADEiso}.
\end{proof}

\end{document}